\documentclass[11pt,twoside]{scrartcl}

\usepackage{dnmath}

\usepackage{amscd}
\usepackage{latexsym}
\usepackage{booktabs}
\usepackage[all]{xy}
\usepackage{verbatim}

\DeclareFontFamily{U}{mathc}{}
\DeclareFontShape{U}{mathc}{m}{it}%
{<->s*[1.03] mathc10}{}
\DeclareMathAlphabet{\mathcal}{U}{mathc}{m}{it}

\usepackage{dnnotation}
\usepackage{burnside-notation}

\theoremstyle{dnplain}

\theoremstyle{dndefinition}

\theoremstyle{dnremark}

\title{Grothendieck Rings of Module Categories over Drinfeld Doubles}
\author{Dmitri Nikshych}
\date{}
\dnshorttitle{Grothendieck Rings over Drinfeld Doubles}
\dnshortauthor{D. Nikshych}
\dnaffiliation{Department of Mathematics and Statistics,\\
University of New Hampshire, Durham, NH 03824, USA}
\dnemail{dmitri.nikshych@unh.edu}

\begin{document}

\begin{abstract}
Let $k$ be an algebraically closed field of characteristic zero and $G$ a
finite group.  We realize the based Grothendieck ring
$R_G$ of $\Rep(D(G))$-module categories  as the degree-two
cocycle-decorated double Burnside ring and derive an explicit Clifford
formula for multiplication and for the action of $R_G$ on the Grothendieck
group of $\Vec_G$-module categories.  We determine the extremal based ideals,
study factorization through smaller groups, and prove a Mackey theorem
for standard two-sided subgroup inductions.  We prove that
$\mathbb C\otimes_{\mathbb Z}R_G$ is semisimple exactly when $G$ is cyclic.
We study the Brauer--Picard action on indecomposable $\Vec_G$-module categories and show that it is
transitive exactly when $G$ is abelian of square-free exponent.
For abelian $G$, we determine the possible abstract group types of
Lagrangian subgroups of $G\oplus\widehat G$, apply the known orthogonal
classification in the homocyclic case, and exhibit same-type
nonconjugate Lagrangians for mixed exponents.
\end{abstract}

\maketitle
\tableofcontents


\section{Introduction}
\label{sect: introduction}

Let $\B$ be a non-degenerate braided fusion category.  Finite semisimple
$\B$-module categories form a monoidal $2$-category under relative tensor
product, and their indecomposable equivalence classes give a distinguished
basis of the Grothendieck ring $K^\mmod(\B)$.  Even when the module
categories themselves are classified, the multiplication and the resulting
ring structure can be difficult to see.  This paper studies that problem
for $\B=\Rep(D(G))\simeq\Z(\Vec_G)$, where $G$ is a finite group and
$D(G)$ is its Drinfeld double.  Our main object is therefore
\begin{equation}
\label{eq: definition RG}
 R_G:=K^\mmod\bigl(\Rep(D(G))\bigr)
 =K_0\bigl(\mathbf{Mod}(\Rep(D(G)))\bigr).
\end{equation}
The main new results concern the explicit multiplication in $R_G$, its
extremal based ideals and factorization structure, its canonical one-sided action,
and Brauer--Picard orbits.  In particular, we determine exactly when the
complexified ring is semisimple.

The categorical description of the basis of $R_G$ is standard.  For any
fusion category $\mathcal C$ there is a $2$-equivalence between the
$2$-category of $\mathcal C$-bimodule categories and the $2$-category of
$\Z(\mathcal C)$-module categories, and this equivalence is monoidal with
respect to relative tensor product; see
\cite[Corollary~2.2]{O} and
\cite[Equation~(17) and Proposition~3.11]{ENO2}.
Taking $\mathcal C=\Vec_G$, module categories over $\Z(\Vec_G)$ are
therefore identified with $\Vec_G$-bimodule categories.
 Indecomposable $\Vec_G$--$\Vec_H$ bimodule categories are described by subgroup--cocycle
pairs $(L,[\mu])$, with $L\leq G\times H$ and
$[\mu]\in H^2(L,k^\times)$; see \cite{O,EGNO,GP}.  We organize these
pairs as transitive degree-two cocycle-decorated bisets and write
$\DBurn(G,H)$ for their Grothendieck group under disjoint union.  This
classification gives basis-preserving isomorphisms
\begin{equation}
\label{eq: introduction decorated model}
 \DBurn(G,H)\xrightarrow{\ \sim\ }
 K_0\bigl(\mathbf{Bimod}(\Vec_G,\Vec_H)\bigr).
\end{equation}
Relative tensor product and its associativity are part of the higher Morita
theory of fusion categories \cite{ENO2,Gr}.
The principal calculation of this paper is the explicit Clifford
decomposition of the relative tensor product in subgroup--cocycle
coordinates.  It supplies composition between these groups and, for
$G=H$, gives
\begin{equation}
\label{eq: introduction RG DBurn}
 R_G\cong\DBurn(G,G)
\end{equation}
as based rings.

For basis elements $[L,\mu]\in\DBurn(G,H)$ and
$[M,\nu]\in\DBurn(H,K)$, write $p_2(L)$ and $p_1(M)$ for their
respective projections to $H$.  With the groups, cocycles, and actions
defined in Subsection~\ref{subsect: definition Clifford composition},
Theorem~\ref{thm: categorical Clifford composition} gives
\begin{equation}
\label{eq: introduction Clifford formula}
 [L,\mu]\circ_H[M,\nu]
 =
 \sum_{g\in p_2(L)\backslash H/p_1(M)}
 \ \sum_{[\pi]\in N_g\backslash\Irr(k_{\alpha_g}[K_g])}
 [N_{g,\pi},\lambda_{g,\pi}].
\end{equation}
The first sum is the ordinary biset double-coset sum.  For each
representative $g$, the second sum runs over the $N_g$-orbits of
isomorphism classes of irreducible $k_{\alpha_g}[K_g]$-modules.
Each orbit contributes the stabilizer $N_{g,\pi}$ of a representative
$[\pi]$, decorated by the cocycle class $[\lambda_{g,\pi}]$ constructed
in that subsection.  Conjugate output pairs are counted with
multiplicity. The new content is the explicit subgroup and cocycle data
of every summand (Theorem~\ref{thm: categorical Clifford composition}).
Galindo--Plavnik's
equivariantization model for relative tensor product \cite[Theorem~7.15 and
Remark~7.17(2)]{GP} is used in the proof.  The formula also specializes to
the canonical action of $R_G$ by taking the right endpoint group to be the trivial group
and the other two groups to be $G$.  Its target is
$\Omega_k^{(2)}(G):=K_0(\mathbf{Mod}(\Vec_G))$
(Corollary~\ref{thm: one-sided Clifford action}); we write $\Omega_G$ for its
distinguished basis.

As an application of the Clifford formula, we recover the invertibility
criterion of Davydov \cite[Corollary~3.6.3]{D}; see also
\cite[Proposition~5.2]{NR}.  A transitive decorated $(G,H)$-biset
$[L,\mu]$ is invertible precisely when $L\leq G\times H$ is subdirect
and its mixed commutator pairing is perfect
(Theorem~\ref{thm: invertibility criterion}).

We next study the ring structure of $R_G$.  The noninvertible basis elements
span a two-sided based ideal $I_G$, which we prove is the unique largest
proper based ideal, with
\begin{equation}
\label{eq: introduction invertible quotient}
 R_G/I_G\cong\mathbb Z[\BrPic(\Vec_G)].
\end{equation}
Thus, unlike in a fusion ring, all noninvertible basis elements belong to a
proper ideal.  At the other end, the ideal $J_G$ of decorated endomorphisms
factoring through the trivial group is the unique minimal nonzero based
ideal.  It has basis $E_{a,b}=a\circ_1b^\vee$, indexed by
$a,b\in\Omega_G$, where $(-)^\vee$ is bimodule duality, defined in
Subsection~\ref{subsect: singular ideal}.  Identify $\DBurn(1,1)$ with
$\mathbb Z$ and define $c_{b,c}$ by $b^\vee\circ_Gc=c_{b,c}$.  Then
\begin{equation}
\label{eq: introduction sandwich ideal}
 E_{a,b}E_{c,d}=c_{b,c}E_{a,d}.
\end{equation}
If $A_b\in\Z(\Vec_G)$ is the Lagrangian algebra corresponding to $b$, then
$c_{b,c}=\dim\Hom_{\Z(\Vec_G)}(A_b,A_c)$, where Hom is taken between the
underlying objects.  The restriction of the one-sided action to $J_G$ is
$E_{a,b}\triangleright c=c_{b,c}a$.  The coefficients form the one-sided
composition matrix $C_G=(c_{b,c})$.  Using these identities, we determine
the radical of the complexified bottom ideal and prove that the full
one-sided action is faithful if and only if $G=1$.

Between the bottom and the top, let $F_{<G}$ be the span of all
distinguished-basis constituents of composites through groups of order
smaller than $|G|$.  For $G\neq1$ one has
\begin{equation}
\label{eq: introduction ideal chain}
 0\subset J_G\subseteq F_{<G}\subseteq I_G\subset R_G.
\end{equation}
We call $I_G/F_{<G}$ the \emph{essential quotient}.  Only noninvertible subdirect
pairs can survive there; the general structure of these subquotients remains
open.
We also decompose the relative tensor product of bimodule categories induced
from $H,K\leq G$ by standard two-sided subgroup induction.  Its summand
indexed by $HgK$ is a relative tensor product over
$\Vec_{g^{-1}Hg\cap K}$.

The main ring-theoretic consequence is that
$\mathbb C\otimes_{\mathbb Z}R_G$ is semisimple if and only if $G$ is
cyclic.
For noncyclic $G$, the complexified bottom ideal has a nonzero Jacobson
radical contained in the radical of $\mathbb C\otimes_{\mathbb Z}R_G$.  For
cyclic $p$-groups, we consider the chain of ideals defined by factorization
through subgroups and prove that the successive quotients are semisimple
after complexification.  This proves semisimplicity of the whole algebra.
The result for arbitrary cyclic groups follows by decomposing the ring as a
tensor product over the Sylow subgroups
(Corollary~\ref{cor: full semisimplicity cyclicity}).

Finally, we study the action of $\BrPic(\Vec_G)$ on $\Omega_G$.  The orbit of
an indecomposable module category $\M$ is sent to the tensor-equivalence class
of its dual category $\Fun_{\Vec_G}(\M,\M)$; by \cite[Proposition~3.6 and Remark~3.7]{MN}, this gives a
bijection with the tensor-equivalence classes of fusion categories in the
Morita class of $\Vec_G$.  We construct a canonical homomorphism
$\BrPic(\Vec_G)\to\Out(G/\Rad_{\mathrm{sol}}(G))$ for which taking
$[K,\nu]$ to the conjugacy class of the image of $K$ in this quotient
is equivariant.  In particular, the abstract quotient $K/\Rad_{\mathrm{sol}}(K)$ is constant
on each orbit.  We also bound the number of orbits below by the number of automorphism
orbits of subgroups of $G/\Rad_{\mathrm{sol}}(G)$
(Corollary~\ref{cor: orbit count solvable quotient}).
When $\Rad_{\mathrm{sol}}(G)=1$, we give an explicit
orbit criterion in terms of automorphisms and restrictions of global
cohomology classes.  The action is transitive exactly when $G$ is abelian
of square-free exponent.

For abelian $G$, the bijection from $\Omega_G$ to the Lagrangian subgroups of
$(G\oplus\widehat G,q)$, where $q(g,\chi)=\chi(g)$, intertwines the
Brauer--Picard action with the orthogonal-group action \cite{ENO2,NN,NR}.
We prove that the possible abstract group types of Lagrangian subgroups
are precisely $\{\operatorname{type}(K\oplus G/K):K\leq G\}$
(Theorem~\ref{thm: possible Lagrangian types}); equivalently, the possible
symmetric equivalence types of Lagrangian subcategories are precisely
$\Rep(K\oplus G/K)$ for $K\leq G$.  We also obtain lower bounds for the
number of orbits for abelian $p$-groups
(Proposition~\ref{prop: abelian orbit count lower bound}).
For homocyclic groups $G=C_{p^n}^k$, Zimmermann's orthogonal classification
\cite{Zimmermann} gives a complete classification of Brauer--Picard orbits
by the abstract group type of the associated Lagrangian subgroup
(Theorem~\ref{thm: homocyclic Lagrangian orbits}), with
$\binom{k+\lfloor n/2\rfloor}{k}$ orbits
(Corollary~\ref{cor: homocyclic orbit count}).
For mixed exponents, abstract group type need not determine the orbit:
Example~\ref{ex: mixed same type nonconjugate Lagrangians} gives two
isomorphic Lagrangian subgroups that are not orthogonally conjugate.

\medskip
\noindent\textbf{Relation to previous work.}
The classical Burnside ring grew out of Burnside's fixed-point calculations
and tables of marks \cite{Burnside}, with its modern ring formulation due to
Solomon \cite{Solomon}.  Cohomological decorations of one-sided $G$-sets
appear in the monomial and crossed Burnside rings \cite{Dr,OY}, in the
cohomological Burnside rings of Hartmann--Yal\c{c}{\i}n \cite{HY}, and in
the decorated-set formalism of Gunnells--Rose--Rumynin \cite{GRR}.  In the
other direction, the biset category and double Burnside rings were developed
by Bouc \cite{Bo1,Bo2}, while the fibered bisets of Boltje--Co\c{s}kun
\cite{BC} supply a degree-one decorated double theory.  For each pair
$(G,H)$, $\DBurn(G,H)$ is, as an abelian group, the degree-two cohomological
Burnside group of $G\times H$.  Our operation is instead the composition
$\DBurn(G,H)\times\DBurn(H,K)\to\DBurn(G,K)$, defined by relative tensor
product rather than Cartesian product of decorated sets.  The comparison
with fibered bisets is by cohomological degree: their degree-one decorations
are replaced here by degree-two cocycle classes.

The ordinary double Burnside algebra is semisimple in characteristic
zero exactly when the group is cyclic, by Bouc's criterion
\cite[Proposition~6.1.7]{Bo2}; see also
\cite[Remark~8.12]{BoltjeDanz}.
Boltje--Danz \cite[Theorem~8.11]{BoltjeDanz} obtain the cyclic
decomposition through inverse-category algebras, while
Bouc--Stancu--Th\'evenaz
\cite[Theorem~10.2 and Corollary~10.3]{BoucStancuThevenaz}
analyze the ideal of maps factoring through the trivial group and its
intersection with the radical.  The corresponding results in this paper involve
the projective/Lagrangian composition matrix and the cocycle-sensitive
factorization arising from Clifford composition.

On the categorical side, the classification of module and bimodule
categories over pointed fusion categories is standard; see
\cite{O,EGNO,GP}.  Relative tensor product belongs to the higher Morita
theory \cite{ENO2,Gr}, and we use the equivariantization model of
Galindo--Plavnik \cite[Theorem~7.15 and Remark~7.17(2)]{GP} to obtain
the explicit subgroup--cocycle Clifford formula.  The invertibility
criterion itself is due to Davydov; see \cite[Corollary~3.6.3]{D}
and the related formulations in \cite{NR,GP}.

For odd $p$, Mason and Ng proved that Lagrangian subgroups
isomorphic to $C_{p^n}^k$ in a fixed nondegenerate form are
conjugate under its isometry group
\cite[Theorem~13.7(ii)]{MasonNg}.
For $G=C_{p^n}^k$, we apply Zimmermann's classification
\cite{Zimmermann} to classify all Lagrangian subgroups of
$G\oplus\widehat G$ up to orthogonal conjugacy, including
$p=2$ (Theorem~\ref{thm: homocyclic Lagrangian orbits}).
Mason and Ng  showed that there exist nonconjugate Lagrangians
of different abstract types \cite[Example~14.1]{MasonNg};
our Example~\ref{ex: mixed same type nonconjugate Lagrangians}
gives nonconjugate Lagrangians of the same abstract type
for mixed exponents.

\medskip
\noindent\textbf{Outline.}
Section~\ref{sect: decorated bisets} develops the decorated-biset model
and proves the Clifford composition formula, including its one-sided
specialization and a derivation of the known invertibility criterion.
Sections~\ref{sect: top and bottom ideals} and~\ref{sect: factorization and Mackey}
turn to the structure of $R_G$: first the extremal based ideals, Gram matrix,
and canonical action, then factorization through smaller groups, the
Mackey theorem, and the cyclic case.
Section~\ref{sect: BrPic orbits Morita class} studies Brauer--Picard orbits,
beginning with general orbit results and transitivity, then passing to
abelian Lagrangian types and orbit counts, the homocyclic case, and the
mixed-exponent counterexample.

\section*{Acknowledgments}

The author is grateful to Alexei Davydov for useful discussions and for
sharing the unpublished work~\cite{ClarkDavydov}.  The author also
acknowledges discussions with ChatGPT, developed by OpenAI, which were
useful in formulating several of the questions, exploring possible
approaches, and assisting with typesetting and manuscript preparation.
The mathematical arguments and decisions concerning their presentation
are the author's own.

The work of the author was supported by the National Science Foundation
under Grant No.~DMS-2302267.


\section{Decorated bisets and Clifford composition}
\label{sect: decorated bisets}
\label{sect: Clifford composition}

Let $k$ be an algebraically closed field of characteristic zero, and let
$G$ and $H$ be finite groups.  We begin with a coordinate-free definition
of a cocycle-decorated biset.  The usual subgroup--cocycle description then
follows from Shapiro's lemma.

\subsection{Bisets and decorations}
\label{subsect: bisets and decorations}

A finite \emph{$(G,H)$-biset} is a finite set $X$ with commuting left
$G$- and right $H$-actions.  We regard it as a left $Q=G\times H$-set by
transporting the usual $G\times H^{\op}$-action along the isomorphism
$\iota_{G,H}:G\times H^{\op}\to G\times H$,
$(g,h^{\op})\mapsto(g,h^{-1})$; thus
\begin{equation}
\label{eq: biset as product group action}
 (g,h)\cdot x=gxh^{-1}.
\end{equation}
Consequently,
$\Stab_Q(x)=\{(g,h)\in G\times H\mid gx=xh\}$, and a transitive biset
with base point $x$ is the $Q$-set $Q/\Stab_Q(x)$.

For a finite $(G,H)$-biset $X$, let $\Fun(X,k^\times)$ be the group of
functions $X\to k^\times$, with pointwise multiplication and the permutation
$Q$-action
\begin{equation}
\label{eq: coefficient action decorated biset}
 \bigl((g,h)\cdot f\bigr)(x)=f(g^{-1}xh).
\end{equation}

\begin{definition}
A \emph{$k^\times$-decorated $(G,H)$-biset}, or simply a
\emph{decorated $(G,H)$-biset}, is a pair $(X,[\xi])$, where
$[\xi]\in H^2\bigl(Q,\Fun(X,k^\times)\bigr)$,
and cohomology is computed for the action
\eqref{eq: coefficient action decorated biset}.
\end{definition}

Choose a normalized cocycle $\xi$ representing $[\xi]$ and write
$\xi(q_1,q_2;x)=\xi(q_1,q_2)(x)$.  Its cocycle identity is
\begin{equation}
\label{eq: decorated cocycle identity}
 \xi(q_2,q_3;q_1^{-1}\!\cdot x)\,
 \xi(q_1,q_2q_3;x)
 =
 \xi(q_1,q_2;x)\,
 \xi(q_1q_2,q_3;x),
\end{equation}
and normalization means
$\xi(1,q;x)=\xi(q,1;x)=1$.

An isomorphism
$\varphi:(X,[\xi_X])\xrightarrow{\sim}(Y,[\xi_Y])$
of decorated $(G,H)$-bisets is a $(G,H)$-equivariant bijection
$\varphi:X\to Y$ such that
$\varphi^*[\xi_Y]=[\xi_X]$ in $H^2\bigl(Q,\Fun(X,k^\times)\bigr)$,
where $\varphi^*f=f\circ\varphi$.

Decorated bisets have disjoint unions.  Since
$\Fun(X\sqcup Y,k^\times)\cong
 \Fun(X,k^\times)\times\Fun(Y,k^\times)$ as $Q$-modules, there is a
canonical decomposition
\[
 H^2\bigl(Q,\Fun(X\sqcup Y,k^\times)\bigr)
 \cong
 H^2\bigl(Q,\Fun(X,k^\times)\bigr)
 \times
 H^2\bigl(Q,\Fun(Y,k^\times)\bigr),
\]
and $(X,[\xi_X])\sqcup(Y,[\xi_Y])$ is decorated by the class
corresponding to $([\xi_X],[\xi_Y])$.

\begin{remark}
The definition and the Shapiro description below work equally well with an
abelian coefficient group $A$ carrying a left $Q$-action.  One gives
$\Fun(X,A)$ the diagonal action
$(q\cdot f)(x)=q\cdot_A f(q^{-1}\!\cdot x)$
and decorates $X$ by a class in $H^2(Q,\Fun(X,A))$.  For transitive
$X\cong Q/L$, evaluation at the base point gives
$H^2(Q,\Fun(X,A))\cong H^2(L,\operatorname{Res}^Q_L A)$.  For the remainder
of the paper, $A=k^\times$ with trivial $Q$-action.
\end{remark}

\subsection{The subgroup--cocycle description}

Let $(X,[\xi])$ be transitive, choose $x\in X$, and put
$L_x=\Stab_Q(x)$.  Since
$\Fun(X,k^\times)\cong\operatorname{Coind}_{L_x}^Q(k^\times)$,
Shapiro's lemma identifies evaluation at $x$ with an isomorphism
\begin{equation}
\label{eq: Shapiro decorated biset}
 \Sh_x:
 H^2\bigl(Q,\Fun(X,k^\times)\bigr)
 \xrightarrow{\ \sim\ }
 H^2(L_x,k^\times),
 \qquad
 [\xi]\longmapsto[\mu_x],
\end{equation}
where, for a normalized representative $\xi$,
$\mu_x(\ell_1,\ell_2)=\xi(\ell_1,\ell_2;x)$.  If $x'=q\cdot x$, then
$L_{x'}=q L_x q^{-1}$ and $[\mu_{x'}]={}^q[\mu_x]$, where
$({}^q\mu)(q\ell q^{-1},qm q^{-1})=\mu(\ell,m)$.  Conversely, the inverse
Shapiro map decorates $Q/L$ from any class in $H^2(L,k^\times)$.
Thus transitive decorated $(G,H)$-bisets are parametrized by
$Q$-conjugacy classes of pairs $(L,[\mu])$, where $L\leq G\times H$ and
$[\mu]\in H^2(L,k^\times)$. We call such a pair a \emph{cocycle-decorated subgroup} of $G\times H$.

Every decorated biset decomposes uniquely, up to isomorphism and order, as a
disjoint union of transitive decorated bisets.  We write $[L,\mu]$ for the
isomorphism class represented by $(L,[\mu])$.

\begin{definition}
The \emph{degree-two decorated Burnside group} from $H$ to $G$ is the
Grothendieck group, under disjoint union, of finite decorated
$(G,H)$-bisets:
\[
 \DBurn(G,H)
 :=K_0^{\sqcup}\bigl(\text{finite decorated $(G,H)$-bisets}\bigr).
\]
Its distinguished basis is indexed by the $G\times H$-conjugacy classes of
pairs $(L,[\mu])$ above.  When $G=H$, we call $\DBurn(G,G)$ the
\emph{degree-two cocycle-decorated double Burnside group} of $G$.
\end{definition}

\subsection{Decorated bisets and bimodule categories}
\label{subsect: twisted linearization and bimodule categories}

\medskip
\noindent
\textbf{Categorical actions.}
Let $Q$ be a finite group, let $Z$ be a finite left $Q$-set, and let
$\beta\in Z^2\bigl(Q,\Fun(Z,k^\times)\bigr)$
be normalized.  Denote by $\V_Q(Z,\beta)$ the category $\Vec_Z$ of
finite-dimensional $Z$-graded vector spaces, equipped with the following
$Q$-action.  For $q\in Q$, put
$(\mathsf T_q U)_z=U_{q^{-1}z}$.
Thus $\mathsf T_q$ sends the simple object indexed by $z$ to the simple
object indexed by $qz$.  The composition isomorphisms of the action are
$a^\beta_{q_1,q_2}:\mathsf T_{q_1}\mathsf T_{q_2} \xrightarrow{\ \sim\ }\mathsf T_{q_1q_2}$
whose $z$-component is multiplication by $\beta(q_1,q_2;z)$.  Its
coherence equation is precisely \eqref{eq: decorated cocycle identity}.  If
$\beta'=\beta\,df$ for a normalized $1$-cochain $f$, the identity functor
$F$ between the two underlying graded categories, with comparison maps
$\mathsf T'_qF\to F\mathsf T_q$ given on the $z$-component by multiplication
by $f(q;z)$, is a $Q$-equivariant equivalence.  Thus the equivariant
equivalence class of $\V_Q(Z,\beta)$ depends only on the decorated $Q$-set
$(Z,[\beta])$.

For $Q=G\times H$, the identification used in
\eqref{eq: biset as product group action} induces the tensor
equivalence
\begin{equation}
\label{eq: tensor equivalence biset convention}
 \Vec_G\boxtimes\Vec_H^{\op}
 \simeq
 \Vec_{G\times H^{\op}}
 \xrightarrow{\ \iota_{G,H}\ }
 \Vec_{G\times H},
 \qquad
 \delta_g\boxtimes\delta_h\longmapsto\delta_{(g,h^{-1})}.
\end{equation}
For a decorated $(G,H)$-biset $\mathbb X=(X,[\xi])$, choose a normalized
representative $\xi$ and write
$\V_{G,H}(\mathbb X)=\V_{G,H}(X,\xi)$
for the resulting bimodule category.  The associated bimodule structure is
given by $\delta_g\triangleright U=\mathsf T_{(g,1)}U$ and
$U\triangleleft\delta_h=\mathsf T_{(1,h^{-1})}U$.  Its associativity and
commuting isomorphisms are induced by the composition isomorphisms $a^\xi$.
This convention is the categorical counterpart of the biset convention
\eqref{eq: biset as product group action}.

\medskip
\noindent
\textbf{Subgroup--cocycle model and classification.}
We now compare this construction with the standard subgroup--cocycle
description of bimodule categories.  For $\mathbb X=(X,[\xi])$, write
$\M_{G,H}(X,[\xi]):=\V_{G,H}(\mathbb X)$.
Suppose first that $(X,[\xi])$ is transitive, choose $x\in X$, put
$L=L_x=\Stab_Q(x)$, and set $[\mu]=\Sh_x([\xi])$.
Choose a normalized representative $\mu\in Z^2(L,k^\times)$ and form the
algebra object
\begin{equation}
\label{eq: subgroup cocycle algebra object}
 A(L,\mu^{-1})=\bigoplus_{\ell\in L}\delta_\ell
 \quad\text{in}\quad
 \Vec_{G\times H},
 \qquad
 \delta_\ell\delta_{\ell'}
 =\mu(\ell,\ell')^{-1}\delta_{\ell\ell'}.
\end{equation}

\begin{proposition}
\label{prop: subgroup cocycle bimodule model}
There is an equivalence of $\Vec_G$--$\Vec_H$ bimodule categories
\begin{equation}
\label{eq: transitive linearization algebra model}
 \M_{G,H}(X,[\xi])
 \simeq
 \operatorname{RMod}_{A(L,\mu^{-1})}(\Vec_{G\times H}).
\end{equation}
The equivalence class is independent of the point $x$, of the representative
$\mu$, and of all auxiliary choices.
\end{proposition}

\begin{proof}
Put $Q=G\times H$.  The simple right
$A(L,\mu^{-1})$-modules are represented by
$\delta_q\otimes A(L,\mu^{-1})$.  Indeed, a nonzero homogeneous vector in a
right module generates a submodule supported on one left coset of $L$, and
simplicity forces the homogeneous components on that coset to be
one-dimensional.  Two such modules are isomorphic exactly when the elements
of $Q$ lie in the same left coset of $L$.  Hence the $Q$-set of simple
objects is $Q/L\cong X$.

For a simple object $S$ fixed up to isomorphism by a subgroup, and choices
$c_\ell:\mathsf T_\ell S\to S$, our stabilizer-cocycle convention is
\[
 c_\ell\circ\mathsf T_\ell(c_m)
 =\gamma(\ell,m)^{-1}c_{\ell m}\circ a_{\ell,m,S}.
\]
At the base point, where $S=A(L,\mu^{-1})$, identify
$\delta_\ell\otimes A(L,\mu^{-1})$ with $A(L,\mu^{-1})$ by multiplication.
For $\ell,m\in L$, the two successive multiplication maps differ from
$c_{\ell m}\circ a_{\ell,m,S}$ by $\mu(\ell,m)^{-1}$.  The preceding rule
therefore gives $\gamma=\mu$.

The algebra model and $\M_{G,H}(X,[\xi])$ have the same $Q$-set of simple
objects.  Write $[\xi']$ for the action class of the algebra model under this
identification.  The calculation gives
$\Sh_x([\xi'])=[\mu]=\Sh_x([\xi])$.  Since $\Sh_x$ is injective,
$[\xi']=[\xi]$.  Rescaling the comparison isomorphisms therefore equips the
equivalence of the underlying semisimple categories with a $Q$-equivariant
structure.  The tensor equivalence
\eqref{eq: tensor equivalence biset convention} gives
\eqref{eq: transitive linearization algebra model}.  Conjugating the base
point or changing $\mu$ by a coboundary produces an equivariantly equivalent
module category.
\end{proof}

When the transitive decorated biset is represented by $(L,[\mu])$, we write
$\M_{G,H}(L,\mu):=\M_{G,H}(X,[\xi])$.
The construction is additive under disjoint unions.  Thus, for a decorated
biset with orbit data $(L_i,[\mu_i])$,
\begin{equation}
\label{eq: linearization orbit decomposition}
 \M_{G,H}(X,[\xi])
 \simeq\bigoplus_{i\in I}\M_{G,H}(L_i,\mu_i).
\end{equation}

Conversely, every finite semisimple $\Vec_G$--$\Vec_H$ bimodule category
arises in this way, and transitive decorated bisets correspond exactly to
indecomposable bimodule categories; see \cite[Appendix~8]{GP} and
\cite{EGNO}.

For a fusion category $\mathcal C$, let $\mathbf{Mod}(\mathcal C)$ denote
 the finite semisimple $2$-category of finite semisimple left
$\mathcal C$-module categories, module functors, and module natural
transformations.

The classification and additivity above give basis-preserving isomorphisms
of abelian groups
\begin{equation}
\label{eq: decorated biset K0 identification}
 \Phi_{G,H}:\DBurn(G,H)
 \xrightarrow{\ \sim\ }
 K_0\bigl(\mathbf{Bimod}(\Vec_G,\Vec_H)\bigr),
 \qquad
 [L,\mu]\longmapsto[\M_{G,H}(L,\mu)].
\end{equation}
Here $K_0$ is the split Grothendieck group on equivalence classes of
objects under direct sum.

When $G=H$, we abbreviate $\M_{G,G}(L,\mu)$ to $\M(L,\mu)$ and write
$\mathbf{Bimod}(\Vec_G) :=\mathbf{Bimod}(\Vec_G,\Vec_G)$.


\subsection{Subgroup--cocycle data for composition}
\label{subsect: definition Clifford composition}

For a subgroup $S\leq P\times Q$ put
\[
\begin{aligned}
 p_1(S)&=\{a\in P\mid (a,b)\in S\text{ for some }b\in Q\},&
 k_1(S)&=\{a\in P\mid (a,1)\in S\},\\
 p_2(S)&=\{b\in Q\mid (a,b)\in S\text{ for some }a\in P\},&
 k_2(S)&=\{b\in Q\mid (1,b)\in S\}.
\end{aligned}
\]
The subgroups $k_1(S)$ and $k_2(S)$ are the coordinate kernels:
$k_1(S)\times1=\ker(p_2|_S)$ and
$1\times k_2(S)=\ker(p_1|_S)$.
Let $G,H,K$ be finite groups, and let $(L,[\mu])$ and $(M,[\nu])$
represent transitive decorated $(G,H)$- and $(H,K)$-bisets, respectively.
Choose normalized representatives $\mu\in Z^2(L,k^\times)$ and
$\nu\in Z^2(M,k^\times)$.
The $G\times K$-orbits in the ordinary balanced product of the underlying
transitive bisets are indexed by
$p_2(L)\backslash H/p_1(M)$.
We now describe the Clifford data associated with a representative $g$ of
such a double coset.

Set $M^g:={}^{(g,1)}M$ and $\nu^g:={}^{(g,1)}\nu$,
and define the fiber-product group
\[
 F_g:=L\times_H M^g
 =\left\{\bigl((a,b),(b,c)\bigr)\ \middle|\
 (a,b)\in L,\ (b,c)\in M^g\right\}.
\]
The homomorphism $q_g:F_g\longrightarrow G\times K$, given by
$q_g\bigl((a,b),(b,c)\bigr)=(a,c)$,
has image
\[
 N_g:=L*M^g
 =\left\{(a,c)\in G\times K\ \middle|\
 (a,b)\in L,\ (b,c)\in M^g\text{ for some }b\in H\right\}
\]
and kernel naturally identified with
$K_g:=k_2(L)\cap k_1(M^g) =k_2(L)\cap{}^g k_1(M)\leq H$.
Thus there is an exact sequence
\[
 1\longrightarrow K_g\longrightarrow F_g
 \xrightarrow{\ q_g\ }N_g\longrightarrow 1,
\]
where $b\in K_g$ is identified with
$\bigl((1,b),(b,1)\bigr)\in F_g$.

Let $\operatorname{pr}_L:F_g\to L$ and
$\operatorname{pr}_M:F_g\to M^g$ be the two projections and put
\[
 \eta_g:=\operatorname{pr}_L^*\mu\,
          \operatorname{pr}_M^*\nu^g
 \in Z^2(F_g,k^\times),
 \qquad
 \alpha_g:=\eta_g|_{K_g}\in Z^2(K_g,k^\times).
\]
Write $\widetilde A_g:=k_{\eta_g}[F_g]$ and
$A_g:=k_{\alpha_g}[K_g]$.
The quotient map $q_g$ makes $\widetilde A_g$ a strongly
$N_g$-graded algebra,
$\widetilde A_g=\bigoplus_{n\in N_g}(\widetilde A_g)_n$, with
$(\widetilde A_g)_1=A_g$.
Indeed, every homogeneous component contains invertible twisted group-basis
elements, so multiplication induces isomorphisms
$(\widetilde A_g)_n\ot_{A_g}(\widetilde A_g)_m \xrightarrow{\ \sim\ }(\widetilde A_g)_{nm}$.
The homogeneous components therefore define a categorical action of $N_g$
on $A_g\text{-}\mathbf{mod}$, and hence an action on
$\Irr(A_g)$:
$n\cdot[\pi] =\left[(\widetilde A_g)_n\ot_{A_g}\pi\right]$.

For $[\pi]\in\Irr(A_g)$ let
$N_{g,\pi}:= \left\{n\in N_g\ \middle|\ (\widetilde A_g)_n\ot_{A_g}\pi\cong\pi\right\}$
be the stabilizer of $[\pi]$ under this action.  Choose
isomorphisms $c_n:(\widetilde A_g)_n\ot_{A_g}\pi
\xrightarrow{\ \sim\ }\pi$ for $n\in N_{g,\pi}$.
Let
\[
 \vartheta_{n,m}:(\widetilde A_g)_n\ot_{A_g}(\widetilde A_g)_m
 \xrightarrow{\ \sim\ }(\widetilde A_g)_{nm}
\]
be the isomorphism induced by multiplication in $\widetilde A_g$.  Define
$\lambda_{g,\pi}(n,m)\in k^\times$ by
\[
c_n\circ(\id\ot c_m) =\lambda_{g,\pi}(n,m)^{-1}\, c_{nm}\circ(\vartheta_{n,m}\ot\id).
\]
Associativity of multiplication in $\widetilde A_g$ gives the cocycle
identity for $\lambda_{g,\pi}$.  Since $\pi$ is simple, any other system of
intertwiners has the form $c'_n=t_n c_n$ for scalars $t_n\in k^\times$, and
then
$\lambda'_{g,\pi}(n,m) =\lambda_{g,\pi}(n,m)\frac{t_{nm}}{t_n t_m}$.
Thus the cohomology class $[\lambda_{g,\pi}]$ is independent of the chosen
intertwiners.  We call $\lambda_{g,\pi}$ a \emph{stabilizer cocycle}.
Its class vanishes exactly when the intertwiners can be rescaled so that
the displayed coherence equation has scalar $1$.
The pair $\bigl(N_{g,\pi},[\lambda_{g,\pi}]\bigr)$ is the output
subgroup--cocycle pair associated with $[\pi]$.

\subsection{The categorical Clifford formula}
\label{subsect: categorical Clifford formula}

\begin{theorem}
\label{thm: categorical Clifford composition}
\label{cor: categorical Clifford formula}
\label{thm: Clifford formula universal product}
Let $(L,[\mu])$ and $(M,[\nu])$ represent transitive decorated
$(G,H)$- and $(H,K)$-bisets, and use the Clifford data of
Subsection~\ref{subsect: definition Clifford composition}.  There is an
equivalence of $\Vec_G$--$\Vec_K$ bimodule categories
\begin{equation}
\label{eqn: categorical Clifford bimodule decomposition}
 \M_{G,H}(L,\mu)\boxtimes_{\Vec_H}\M_{H,K}(M,\nu)
 \simeq
 \bigoplus_{\substack{
  g\in p_2(L)\backslash H/p_1(M)\\
  [\pi]\in N_g\backslash\Irr(A_g)}}
 \M_{G,K}(N_{g,\pi},\lambda_{g,\pi}).
\end{equation}
The decomposition is independent of the choices of cocycle representatives,
double-coset representatives, orbit representatives, and intertwiners, up
to the usual equivalence; conjugate repeated outputs are retained with
multiplicity.
\end{theorem}

\begin{proof}
Let $X$ and $Y$ be the transitive underlying $(G,H)$- and $(H,K)$-bisets,
with decorations represented by $\xi$ and $\eta$.  Put
\[
 W=X\times Y,
 \qquad
 P=G\times H\times K,
 \qquad
 H_{\mathrm m}=1\times H\times1,
 \qquad
 R=P/H_{\mathrm m}\cong G\times K,
\]
where
$(a,b,c)\cdot(x,y)=\bigl(axb^{-1},byc^{-1}\bigr)$.
The external-product cocycle on this action is
\[
 \Omega\bigl((a_1,b_1,c_1),(a_2,b_2,c_2);x,y\bigr)
 =\xi\bigl((a_1,b_1),(a_2,b_2);x\bigr)
  \eta\bigl((b_1,c_1),(b_2,c_2);y\bigr).
\]
By \cite[Theorem~7.15 and Remark~7.17(2)]{GP}, specialized to trivial
associators, there is an equivalence of $\Vec_G$--$\Vec_K$ bimodule
categories
\begin{equation}
\label{eqn: GP middle equivariantization}
 \M_{G,H}(X,[\xi])\boxtimes_{\Vec_H}\M_{H,K}(Y,[\eta])
 \simeq
 \bigl(\Vec_W,\Omega\bigr)^{H_{\mathrm m}},
\end{equation}
where the right-hand side is the equivariantization by $H_{\mathrm m}$ with its residual
$R$-action.

Choose base points $x\in X$ and $y\in Y$ with stabilizers $L$ and $M$.
The $R$-orbits of $H_{\mathrm m}$-orbits in $W$ are indexed by
$p_2(L)\backslash H/p_1(M)$.
For a representative $g$, use the point $(x,gy)$.
Identify $F_g=L\times_HM^g$ with a subgroup of $P$ by
$((a,b),(b,c))\mapsto(a,b,c)$.
Then $F_g$ is its stabilizer in $P$, $K_g=k_2(L)\cap k_1(M^g)$
is its $H_{\mathrm m}$-stabilizer, and $N_g=L*M^g$ is the
stabilizer in $R$ of its $H_{\mathrm m}$-orbit.
Recalling that $q_g\bigl((a,b),(b,c)\bigr)=(a,c)$, the stabilizer sequence is
\[
 1\longrightarrow K_g\longrightarrow F_g
 \xrightarrow{q_g}N_g\longrightarrow1.
\]
Restriction to $F_g$ followed by evaluation at $(x,gy)$ gives
\[
 \bigl[(f,f')\longmapsto\Omega(f,f';x,gy)\bigr]
 =\bigl[\operatorname{pr}_L^*\mu\,
        \operatorname{pr}_M^*\nu^g\bigr]
 =[\eta_g]\quad\text{in }H^2(F_g,k^\times).
\]
For this component, we use the representative $\eta_g$, after the
corresponding rescaling of equivariant structures.

The $H_{\mathrm m}$-stabilizer at $(x,gy)$ is $K_g$, with projective cocycle
$\alpha_g=\eta_g|_{K_g}$.  Evaluation at $(x,gy)$ sends an
$H_{\mathrm m}$-equivariant object supported on its orbit to a vector space
$V$ with operators $\rho(b)$, $b\in K_g$, satisfying
$\rho(b)\rho(b')=\alpha_g(b,b')\rho(bb')$.  Thus the fiber is a left
$A_g$-module, where $A_g=k_{\alpha_g}[K_g]$.  Conversely, the module
determines an equivariant object by transport along the transitive orbit.
Evaluation therefore gives an equivalence with
$A_g\text{-}\mathbf{mod}$.  Its simple objects are indexed by $\Irr(A_g)$.

Under this identification, the residual action of $N_g$ on
$A_g\text{-}\mathbf{mod}$ is given by the strongly graded algebra
$\widetilde A_g=k_{\eta_g}[F_g] =\bigoplus_{n\in N_g}(\widetilde A_g)_n$, with
$(\widetilde A_g)_1=A_g$.
Multiplication gives the bimodule isomorphisms
\[
 \vartheta_{n,m}:
 (\widetilde A_g)_n\otimes_{A_g}(\widetilde A_g)_m
 \xrightarrow{\sim}(\widetilde A_g)_{nm},
\]
and transport by $n$ is
$T_n(\pi)=(\widetilde A_g)_n\otimes_{A_g}\pi$.
To see this directly, let $f=(a,b,c)\in F_g$ lift $n=(a,c)$.  This lift
fixes $(x,gy)$, so
$(a,1,c)\cdot(x,gy)=(1,b^{-1},1)\cdot(x,gy)$.
Thus $(a,1,c)$ moves the point within its $H_{\mathrm m}$-orbit,
and $(1,b,1)$ brings it back; the equivariant structure supplies
the corresponding identification of fibers.  If $u_f$ is the twisted
group-basis unit, then $(\widetilde A_g)_n=u_fA_g$.  Under
$u_f\otimes v\leftrightarrow v$, an element $a\in A_g$ acts as
$\pi(u_f^{-1}au_f)$.  This is the transported action on the equivariant
fiber; multiplication supplies the composition maps $\vartheta_{n,m}$.
The homogeneous component is intrinsic: changing the lift or the chosen
unit changes this presentation by an isomorphic functor.

Over the $G\times K$-orbit of the chosen $H_{\mathrm m}$-orbit, the
simple-object set is isomorphic, as a $G\times K$-set, to
$(G\times K)\times_{N_g}\Irr(A_g)$.  Consequently, the $G\times K$-orbits
are indexed by $N_g\backslash\Irr(A_g)$.

Fix $[\pi]\in\Irr(A_g)$.  The point represented by $(1,[\pi])$ has
stabilizer
$N_{g,\pi} =\{n\in N_g\mid (\widetilde A_g)_n\otimes_{A_g}\pi\simeq\pi\}$.
Choose intertwiners $c_n:(\widetilde A_g)_n\otimes_{A_g}\pi
\xrightarrow{\sim}\pi$ for $n\in N_{g,\pi}$.
The equation
$c_n\circ(\id\otimes c_m) =\lambda_{g,\pi}(n,m)^{-1} c_{nm}\circ(\vartheta_{n,m}\otimes\id)$
says precisely that $\lambda_{g,\pi}$ is the stabilizer cocycle of the
simple object $\pi$ for the residual $G\times K$-action.  

Let $\mathcal C_{g,\pi}$ be the full $G\times K$-stable semisimple
subcategory generated by the orbit of $\pi$.  Its simple objects form the
transitive set
$(G\times K)/N_{g,\pi}$,
and the coherence at the base point is the preceding intertwiner equation.
Hence
$\mathcal C_{g,\pi} \simeq\M_{G,K}(N_{g,\pi},\lambda_{g,\pi})$
as $\Vec_G$--$\Vec_K$ bimodule categories by
Proposition~\ref{prop: subgroup cocycle bimodule model}.  Decomposition into
$G\times K$-orbits of simple objects gives
\[
 \bigl(\Vec_W,\Omega\bigr)^{H_{\mathrm m}}
 \simeq
 \bigoplus_{\substack{
  g\in p_2(L)\backslash H/p_1(M)\\
  [\pi]\in N_g\backslash\Irr(A_g)}}
 \M_{G,K}(N_{g,\pi},\lambda_{g,\pi}).
\]
Together with \eqref{eqn: GP middle equivariantization}, this proves
\eqref{eqn: categorical Clifford bimodule decomposition}.

Finally, write $g'=h_1gh_2$ with $h_1\in p_2(L)$ and
$h_2\in p_1(M)$, and choose $(a,h_1)\in L$ and $(h_2,c)\in M$.  Then
$(a,h_1,c^{-1})\cdot(x,gy)=(x,g'y)$.  Thus $F_{g'}$ is conjugate to $F_g$
in $P$ by $(a,h_1,c^{-1})$, while the output subgroup--cocycle pairs are
conjugated in $G\times K$ by $(a,c^{-1})$.  Replacing $[\pi]$ by
$n\cdot[\pi]$ changes its stabilizer to
$nN_{g,\pi}n^{-1}$ and transports its cocycle class by this conjugation.
Replacing
$c_n$ by $t_nc_n$ changes $\lambda_{g,\pi}$ by a coboundary, and changing
the input cocycles gives equivalent twisted actions.  Repeated conjugate
outputs are retained with their multiplicities.  This also proves the
asserted independence of choices.
\end{proof}

We define the \emph{Clifford composition} of basis elements by
\begin{equation}
\label{eqn: decorated double Burnside product}
 [L,\mu]\circ_H[M,\nu]
 =
 \sum_{g\in p_2(L)\backslash H/p_1(M)}
 \ \sum_{[\pi]\in N_g\backslash\Irr(A_g)}
 [N_{g,\pi},\lambda_{g,\pi}],
\end{equation}
and extend it bilinearly.

\begin{corollary}
\label{cor: decorated bisets and bimodule categories}
The compositions $\circ_H$ are associative and have identity
$[\Delta(G),1]$ at $G$.  They make the groups $\DBurn(G,H)$ into a based
$\mathbb Z$-linear category with
$\operatorname{Hom}(H,G)=\DBurn(G,H)$, and the maps
$\Phi_{G,H}:\DBurn(G,H) \xrightarrow{\ \sim\ } K_0\bigl(\mathbf{Bimod}(\Vec_G,\Vec_H)\bigr)$
intertwine Clifford composition with relative tensor product.
\end{corollary}

\begin{proof}
By Theorem~\ref{thm: categorical Clifford composition} and additivity,
\[
 \Phi_{G,K}(x\circ_Hy)
 =\Phi_{G,H}(x)\boxtimes_{\Vec_H}\Phi_{H,K}(y),
 \qquad
 x\in\DBurn(G,H),\quad y\in\DBurn(H,K).
\]
Associativity therefore follows from associativity of
relative tensor product; see \cite[Proposition~4.4]{Gr}.
The regular bimodule category $\Vec_G$ is its unit by
\cite[Proposition~3.15]{Gr}, and
$\M_{G,G}(\Delta(G),1)\simeq\Vec_G$.
Since the maps $\Phi_{G,H}$ are isomorphisms, this gives associativity
and the stated identities for Clifford composition.
\end{proof}

\begin{corollary}
\label{cor: decorated bisets Drinfeld double}
For every finite group $G$, there is a canonical ring isomorphism
$K^\mmod\bigl(\Rep(D(G))\bigr)\cong\DBurn(G,G)$,
where $K^\mmod(\B)$ denotes the Grothendieck ring of finite semisimple
module categories over a braided fusion category $\B$.
\end{corollary}

\begin{proof}
Put $\mathcal C=\Vec_G$.  The assignment
$\M\longmapsto\Fun_{\mathcal C\mid\mathcal C}(\mathcal C,\M)$, with
bimodule functors and bimodule natural transformations, gives the monoidal
$2$-equivalence by \cite[Proposition~7.12 and Theorem~7.14]{Gr}.  Its
tensor-product compatibility is
\[
 \Fun_{\mathcal C\mid\mathcal C}
  (\mathcal C,\M\boxtimes_{\mathcal C}\N)
 \simeq
 \Fun_{\mathcal C\mid\mathcal C}(\mathcal C,\M)
 \boxtimes_{\Z(\mathcal C)}
 \Fun_{\mathcal C\mid\mathcal C}(\mathcal C,\N).
\]
Consequently, the basis-preserving map from $\DBurn(G,G)$ sends $[L,\mu]$ to
the class of
\[
 \Fun_{\mathcal C\mid\mathcal C}(\mathcal C,\M_{G,G}(L,\mu))
\]
and preserves products by the preceding corollary.  Transporting along
$\Z(\Vec_G)\simeq_{\mathrm{br}}\Rep(D(G))$ finishes the proof.
\end{proof}

\begin{corollary}
\label{cor: Clifford formula double Burnside ring}
For every finite group $G$, the group $\DBurn(G,G)$ is a based ring with
basis indexed by the $(G\times G)$-conjugacy classes of pairs
$(L,[\mu])$, where $L\leq G\times G$ and $[\mu]\in H^2(L,k^\times)$.
Its unit is $[\Delta(G),1]$, and the product of basis elements is
\begin{equation}
\label{eqn: endomorphism decorated double Burnside product}
 [L,\mu]\,[M,\nu]
 =
 \sum_{g\in p_2(L)\backslash G/p_1(M)}
 \ \sum_{[\pi]\in N_g\backslash\Irr(A_g)}
 [N_{g,\pi},\lambda_{g,\pi}].
\end{equation}
Under $\Phi_{G,G}$, this is the multiplication of
$K_0(\mathbf{Bimod}(\Vec_G))$:
\begin{equation}
\label{eqn: K0 Bimod Clifford product}
 [\M(L,\mu)]\,[\M(M,\nu)]
 =
 \sum_{g\in p_2(L)\backslash G/p_1(M)}
 \ \sum_{[\pi]\in N_g\backslash\Irr(A_g)}
 [\M(N_{g,\pi},\lambda_{g,\pi})].
\end{equation}
\end{corollary}

\begin{proof}
This is Theorem~\ref{thm: categorical Clifford composition} and
Corollary~\ref{cor: decorated bisets and bimodule categories}, specialized
to $G=H=K$.
\end{proof}

For a cocycle $\alpha$ and commuting elements $x,y$ of its group, put
$\Alt_\alpha(x,y)=\alpha(x,y)\alpha(y,x)^{-1}$.

\begin{lemma}
\label{lem: twisted overlap simple modules}
Let $C$ be a finite abelian group, let $\alpha\in Z^2(C,k^\times)$,
and put $\mathcal R=\Rad(\Alt_\alpha)$.  Choosing a cochain
trivializing $\alpha|_{\mathcal R}$ identifies $\Irr(k_\alpha[C])$
with $\widehat{\mathcal R}$.  The natural action of $\widehat C$ on
$\Irr(k_\alpha[C])$ is transitive, and the stabilizer of every simple
module consists of the characters trivial on $\mathcal R$.
\end{lemma}

\begin{proof}
The relation $u_cu_d=\Alt_\alpha(c,d)u_du_c$ shows that the center
of $k_\alpha[C]$ is $\bigoplus_{r\in\mathcal R}ku_r$.
After choosing a cochain trivializing $\alpha|_{\mathcal R}$,
rescaling these basis elements identifies the center with
$k[\mathcal R]$.  Since $k_\alpha[C]$ is semisimple, its simple
modules are indexed by $\widehat{\mathcal R}$.  Pullback along
$u_c\mapsto\chi(c)u_c$, for $\chi\in\widehat C$, multiplies the
central character by $\chi|_{\mathcal R}$.  Every character of
$\mathcal R$ extends to $C$, proving the assertions.
\end{proof}

\begin{corollary}
\label{cor: abelian Clifford formula}
Let $G,H,K$ be finite abelian groups, and let $(L,[\mu])$ and
$(M,[\nu])$ be decorated subgroups of $G\times H$ and $H\times K$,
respectively.  Choose normalized representatives $\mu$ and $\nu$, and put
$K_1=k_2(L)\cap k_1(M)$, $N_1=L*M$, and $A_1=k_{\alpha_1}[K_1]$,
where
$\alpha_1(b,b') =\mu\bigl((1,b),(1,b')\bigr) \nu\bigl((b,1),(b',1)\bigr)$.
For $[\pi]\in\Irr(A_1)$, let
$N=N_{1,\pi}$ be the stabilizer of $[\pi]$ in $N_1$ and let
$[\lambda]=[\lambda_{1,\pi}]$ be its associated cocycle class, as in
Subsection~\ref{subsect: definition Clifford composition}.  The resulting
basis element $[N,\lambda]$ is independent of $[\pi]$, and
\begin{equation}
\label{eqn: abelian Clifford formula}
 [L,\mu]\circ_H[M,\nu]
 = [H:p_2(L)p_1(M)]\,
   \bigl|N_1\backslash\Irr(A_1)\bigr|\,[N,\lambda],
\end{equation}
so the product of two basis elements is monomial.
\end{corollary}

\begin{proof}
Since $H$ is abelian, the double cosets are the cosets of
$p_2(L)p_1(M)$, and conjugation by a representative does not change any of
the Clifford data.  Thus the double-coset sum in
\eqref{eqn: decorated double Burnside product} contributes the factor
$[H:p_2(L)p_1(M)]$.

To compare the simple $A_1$-modules and their $N_1$-orbits, recall that
$F_1=L\times_HM$ and
$\widetilde A_1=k_{\operatorname{pr}_L^*\mu\,
\operatorname{pr}_M^*\nu}[F_1]$; this is the strongly $N_1$-graded
algebra whose degree-one part is $A_1$.  The group $F_1$ is abelian.
The character group $\widehat K_1$ acts
transitively on $\Irr(A_1)$ by tensor product; see
Lemma~\ref{lem: twisted overlap simple modules}.  Every
character of $K_1$ extends to $F_1$.  If $\widetilde\chi$ is such an
extension, the grading-preserving automorphism
$u_f\longmapsto\widetilde\chi(f)u_f$ for $f\in F_1$ induces the character
twist $\pi\mapsto\chi\otimes\pi$ on $A_1$-modules and is compatible with
the composition isomorphisms of the $N_1$-action.  Hence
$N_{1,\chi\otimes\pi}=N_{1,\pi}$ and
$[\lambda_{1,\chi\otimes\pi}]=[\lambda_{1,\pi}]$, where the latter equality
is in the cohomology of this common stabilizer.  Thus all $N_1$-orbits in
$\Irr(A_1)$ give the same decorated basis element.  Counting the orbits gives
\eqref{eqn: abelian Clifford formula}.
\end{proof}

\begin{remark}
For tensor products of pointed bimodule categories, see
\cite[Proposition~3.19]{ENO2}; \cite[Theorem~10.4]{ENO2} gives a description
using Lagrangian correspondences.
\end{remark}

\begin{example}
\label{ex: decorated double Burnside C2}
Let $C_2=\langle s\rangle$ and put $V=C_2\times C_2$.  The cyclic
subgroups have trivial second cohomology with coefficients in $k^\times$,
whereas
$H^2(V,k^\times)\cong C_2$.
Let $[\beta]$ be its nontrivial class.  We may choose the normalized
representative
$\beta\bigl((s^a,s^b),(s^{a'},s^{b'})\bigr)=(-1)^{ab'}$ for
$a,b,a',b'\in\mathbb F_2$,
whose alternating bicharacter is
$\Alt_\beta\bigl((s^a,s^b),(s^{a'},s^{b'})\bigr) =(-1)^{ab'-a'b}$.
The six basis elements of $\DBurn(C_2,C_2)$ are
\[
\begin{gathered}
 \mathbf 1=[\Delta(C_2),1],\qquad
 \omega=[V,\beta],\qquad
 q=[V,1],\\
 \ell=[C_2\times1,1],\qquad
 r=[1\times C_2,1],\qquad
 e=[1,1].
\end{gathered}
\]
With rows acting on the left and columns on the right, formula
\eqref{eqn: endomorphism decorated double Burnside product} gives
\[
\renewcommand{\arraystretch}{1.18}
\begin{array}{c|cccccc}
 \cdot & \mathbf 1 & \omega & q & \ell & r & e \\ \hline
 \mathbf 1 & \mathbf 1 & \omega & q & \ell & r & e \\
 \omega    & \omega    & \mathbf 1 & r & e & q & \ell \\
 q         & q         & \ell & 2q & 2\ell & q & \ell \\
 \ell      & \ell      & q & q & \ell & 2q & 2\ell \\
 r         & r         & e & 2r & 2e & r & e \\
 e         & e         & r & r & e & 2r & 2e
\end{array}
\]
Two entries illustrate the Clifford mechanism.  First consider $q^2$.
There is one double coset, and
$F_1\cong C_2^3$, $K_1\cong C_2$, $N_1\cong V$, and
$A_1=k[K_1]\cong k\oplus k$.
The action of $N_1$ on the two simple $A_1$-modules is trivial.  Each simple
module therefore forms a singleton orbit, with stabilizer $N_1$ and
trivial cocycle class.  Both orbits produce the same output $q$, so
$q^2=2q$.

For $\omega^2$ the groups $F_1$, $K_1$, and $N_1$ are the same, and the
restriction of the product cocycle to $K_1$ is again trivial; hence
$A_1\cong k[C_2]$ still has the two simples
$\chi_t(s^b)=(-1)^{tb}$ for $t\in\mathbb F_2$.
The difference is the residual action.  The homogeneous component of degree
$(s^a,s^c)\in N_1=V$ sends $\chi_t$ to
$\chi_{t+a+c}$.
Consequently $N_1$ acts transitively on the two simples.  The stabilizer of a
simple-module class is $\Delta(C_2)$, and its cocycle class is trivial.
There is therefore a single summand and
$\omega^2=\mathbf 1$.
The element $\omega$ exchanges the regular and rank-one
$\Vec_{C_2}$-module categories under the one-sided action; see
Example~\ref{ex: one-sided decorated Burnside C2 action}.

The table also displays a feature absent from the one-sided cohomological
Burnside ring: the double ring is noncommutative.  Indeed,
$q\omega=\ell$ and $\omega q=r$.
By Theorem~\ref{thm: categorical Clifford composition}, the same table is the
multiplication table of $K_0(\mathbf{Bimod}(\Vec_{C_2}))$ after replacing
each symbol $[L,\mu]$ by $[\M(L,\mu)]$.

For $G=S_3$ the rank is already $28$.  Ostrik enumerated the corresponding
indecomposable module categories over
$\Rep(D(S_3))\simeq\Z(\Vec_{S_3})$ by listing the conjugacy classes
of subgroups of $S_3\times S_3$ together with their cohomology classes; see
\cite[Sections~4.1--4.2]{O}.  A complete multiplication table was subsequently
computed by Clark and Davydov \cite{ClarkDavydov}.  The $C_2$ case above is
the smallest example in which the Clifford correction is visible.
\end{example}

\subsection{The action on one-sided decorated Burnside rings and module categories}
\label{sect: one-sided action}

Since $\Vec_1=\Vec$, a $(\Vec_G,\Vec_1)$-bimodule category is simply a
left $\Vec_G$-module category.  Indecomposable such categories are
parametrized by conjugacy classes of pairs $(H,[\alpha])$, where $H\leq G$
and $[\alpha]\in H^2(H,k^\times)$; see Example~2.1 of \cite{O}.  For a normalized
cocycle $\alpha\in Z^2(H,k^\times)$, write
\[
 \M(H,\alpha):=\M_{G,1}(H\times1,\alpha)
 =\operatorname{RMod}_{A(H,\alpha^{-1})}(\Vec_G),
 \qquad
 A(H,\alpha^{-1})=\bigoplus_{h\in H}\delta_h,
\]
where multiplication in $A(H,\alpha^{-1})$ is twisted by $\alpha^{-1}$,
in accordance with the convention in
\eqref{eq: subgroup cocycle algebra object}.  Put
$\OneBurn(G) :=\bigoplus_{[H,\alpha]}\mathbb Z[H,\alpha]$.
The identification $\OneBurn(G)\xrightarrow{\ \sim\ }\DBurn(G,1)$ given by
$[H,\alpha]\longmapsto[H\times1,\alpha]$,
and Corollary~\ref{cor: decorated bisets and bimodule categories} give the
basis-preserving isomorphism
\[
 \Phi_\Omega:=\Phi_{G,1}:\OneBurn(G)
 \xrightarrow{\ \sim\ }
 K_0\bigl(\mathbf{Mod}(\Vec_G)\bigr),
 \qquad
 [H,\alpha]\longmapsto[\M(H,\alpha)].
\]
Write
$\Phi_B:=\Phi_{G,G}:\DBurn(G,G) \xrightarrow{\ \sim\ } K_0\bigl(\mathbf{Bimod}(\Vec_G)\bigr)$
for the ring isomorphism of
Corollary~\ref{cor: decorated bisets and bimodule categories}.
Under the same identification the natural action is the specialization of
Clifford composition $\DBurn(G,G)\times \DBurn(G,1) \longrightarrow
\DBurn(G,1)$, where $x\triangleright y:=x\circ_G y$.
Thus the one-sided formula is the specialization of
\eqref{eqn: decorated double Burnside product} to
$(G,H,K)=(G,G,1)$.

We record the specialized group-theoretic data explicitly.  Let
$(L,[\mu])$ be a cocycle-decorated subgroup of $G\times G$, let
$(H,[\alpha])$ be a cocycle-decorated subgroup of $G$, and choose normalized
representatives $\mu$ and $\alpha$.  For
$g\in p_2(L)\backslash G/H$,
the fiber product from Subsection~\ref{subsect: definition Clifford composition}
identifies with
$F_g=\{(a,b)\in L\mid b\in{}^gH\}$.
Under this identification the exact sequence and product cocycle are
\[
 1\longrightarrow K_g\longrightarrow F_g
 \xrightarrow{\ q_g\ }N_g\longrightarrow1,
 \qquad q_g(a,b)=a,
\]
where $K_g\cong k_2(L)\cap{}^gH$ and
$N_g=p_1\bigl(L\cap(G\times{}^gH)\bigr)\leq G$,
and
$\eta_g=\mu|_{F_g}\,\operatorname{pr}_2^*({}^g\alpha) \in Z^2(F_g,k^\times)$.
Let
$A_g:=k_{\eta_g|_{K_g}}[K_g]$,
and, for $[\pi]\in\Irr(A_g)$, let $N_{g,\pi}\leq N_g$ be its stabilizer and
$\lambda_{g,\pi}\in Z^2(N_{g,\pi},k^\times)$ its stabilizer cocycle
obtained from the strongly $N_g$-graded algebra
$k_{\eta_g}[F_g]$ exactly as in Subsection~\ref{subsect: definition Clifford composition}.

\begin{corollary}
\label{thm: one-sided Clifford action}
With the notation above, the following hold.
\begin{enumerate}
\item[\textup{(i)}] \textbf{Burnside form.}
For basis elements of $\DBurn(G,G)$ and $\OneBurn(G)$,
\begin{equation}
\label{eqn: one-sided Clifford action formula}
 [L,\mu]\triangleright[H,\alpha]
 =
 \sum_{g\in p_2(L)\backslash G/H}
 \ \sum_{[\pi]\in N_g\backslash\Irr(A_g)}
 [N_{g,\pi},\lambda_{g,\pi}].
\end{equation}
The second sum is over representatives of the $N_g$-orbits, and conjugate
output pairs are combined with multiplicity.  Formula
\eqref{eqn: one-sided Clifford action formula} is independent of all choices
and makes $\OneBurn(G)$ a unital left $\DBurn(G,G)$-module.

\item[\textup{(ii)}] \textbf{Categorical form.}
The basis-preserving isomorphism
$\Phi_\Omega:\OneBurn(G) \xrightarrow{\ \sim\ } K_0\bigl(\mathbf{Mod}(\Vec_G)\bigr)$
and the ring isomorphism $\Phi_B$ identify the Burnside action with the
categorical action.  Explicitly,
\[
 \Phi_\Omega(x\triangleright y)
 =\Phi_B(x)\bt_{\Vec_G}\Phi_\Omega(y),
 \qquad
 x\in \DBurn(G,G),\quad y\in\OneBurn(G),
\]
and, for basis elements,
\[
 \M(L,\mu)\bt_{\Vec_G}\M(H,\alpha)
 \simeq
 \bigoplus_{\substack{
 g\in p_2(L)\backslash G/H\\
 [\pi]\in N_g\backslash\Irr(A_g)}}
 \M(N_{g,\pi},\lambda_{g,\pi}).
\]
\end{enumerate}
\end{corollary}

\begin{proof}
Apply Theorem~\ref{thm: categorical Clifford composition}, equivalently
formula~\eqref{eqn: decorated double Burnside product}, to
$(L,\mu)\in \DBurn(G,G)$ and
$(H\times1,\alpha)\in \DBurn(G,1)$.  The double cosets and all
Clifford data specialize to those displayed above, giving
\eqref{eqn: one-sided Clifford action formula}.  The categorical statement
is the case $(G,H,K)=(G,G,1)$ of
Theorem~\ref{thm: categorical Clifford composition}.  Unitality and module
associativity follow from
Corollary~\ref{cor: decorated bisets and bimodule categories}.
\end{proof}

\begin{example}
\label{ex: one-sided decorated Burnside C2 action}
Continue with the notation of
Example~\ref{ex: decorated double Burnside C2}.  The one-sided decorated
Burnside ring has basis $\xi=[1,1]$ and $\tau=[C_2,1]$.
Writing the acting element horizontally, the action of
$\DBurn(C_2,C_2)$ on $\OneBurn(C_2)$ is
\[
\renewcommand{\arraystretch}{1.18}
\begin{array}{c|cccccc}
 x & \mathbf 1 & \omega & q & \ell & r & e \\ \hline
 x\triangleright\xi
   & \xi & \tau & \tau & 2\tau & \xi & 2\xi \\
 x\triangleright\tau
   & \tau & \xi & 2\tau & \tau & 2\xi & \xi
\end{array}
\]
Categorically, $\xi$ corresponds to the regular module category
$\M(1,1)\simeq\Vec_{C_2}$,
whereas $\tau$ corresponds to the rank-one module category
$\M(C_2,1)\simeq\Vec$.  Thus the table records the relative tensor products
of these two module categories with the six indecomposable
$\Vec_{C_2}$-bimodule categories.

For example, consider the noninvertible bimodule $\M(V,1)$.  Since
$q\triangleright\xi=\tau$, associativity and $q^2=2q$ give
$q\triangleright\tau=(q^2)\triangleright\xi=2\tau$.  On the other hand, the
Fourier bimodule
$\M(V,\beta)$ is invertible and exchanges the two module categories:
$\M(V,\beta)\bt_{\Vec_{C_2}}\Vec_{C_2}\simeq\Vec$ and
$\M(V,\beta)\bt_{\Vec_{C_2}}\Vec\simeq\Vec_{C_2}$.
\end{example}

\begin{remark}
For $L=\Delta(G)$ there is one double coset, $K_g=1$, $N_g=H$, and the
output cocycle is $\alpha$.  Thus
$[\Delta(G),1]\triangleright[H,\alpha]=[H,\alpha]$.
More generally, for a diagonal decorated subgroup
$[\Delta(K),\beta]$ the groups $K_g$ are trivial, and
\eqref{eqn: one-sided Clifford action formula} reduces to
\[
 [\Delta(K),\beta]\triangleright[H,\alpha]
 =\sum_{g\in K\backslash G/H}
 \left[K\cap{}^gH,
 \beta|_{K\cap{}^gH}\,{}^g\alpha|_{K\cap{}^gH}\right].
\]
This is the usual multiplication formula in the one-sided cohomological
Burnside ring.  The map $[H,\alpha]\mapsto[\Delta(H),\alpha]$ is an
injective unital ring homomorphism, sending $[G,1]$ to
$[\Delta(G),1]$.  Injectivity follows because decorated diagonal pairs are
conjugate exactly when the original subgroup--cocycle pairs are
$G$-conjugate.  Applying the Clifford formula to two decorated diagonal
subgroups gives the corresponding decorated diagonals with the one-sided
product coefficients, since their coordinate kernels are trivial.  Under
this embedding, the action restricts to multiplication in the one-sided
cohomological Burnside ring.
\end{remark}

\subsection{Invertibility and Morita equivalence}
\label{subsect: invertibility criterion}

Let $G$ and $H$ be finite groups.  A decorated $(G,H)$-biset $\mathbb X$ is
\emph{invertible} if there is a decorated $(H,G)$-biset $\mathbb Y$ such
that $[\mathbb X]\circ_H[\mathbb Y]=[\Delta(G),1]$ and
$[\mathbb Y]\circ_G[\mathbb X]=[\Delta(H),1]$.  Here $\mathbb Y$ is required
to be an actual decorated biset.
Under the identifications $\Phi_{G,H}$ and $\Phi_{H,G}$ of
Corollary~\ref{cor: decorated bisets and bimodule categories}, these
equations are equivalent to
\[
 \V_{G,H}(\mathbb X)\bt_{\Vec_H}\V_{H,G}(\mathbb Y)\simeq\Vec_G,
 \qquad
 \V_{H,G}(\mathbb Y)\bt_{\Vec_G}\V_{G,H}(\mathbb X)\simeq\Vec_H.
\]
Indeed, equality of classes of actual semisimple categories determines their
indecomposable decompositions.  Thus an invertible decorated biset implements
a categorical Morita equivalence
$\Vec_G\simeq_{\mathrm{Mor}}\Vec_H$,
and hence braided equivalences  $\Rep(D(G))\simeq_{\mathrm{br}}\Rep(D(H))$.

A product of two nonempty transitive decorated bisets is nonzero: its
Clifford formula has a double coset and a simple module of a nonzero
semisimple twisted group algebra.  By positivity, a decomposition of either
factor into more than one transitive component could not give a single
identity term.  Hence an invertible decorated biset and its actual inverse
are transitive.  Represent the first by a pair $(L,[\mu])$, where
$L\leq G\times H$ and
$[\mu]\in H^2(L,k^\times)$.  Put $A=k_1(L)$ and $B=k_2(L)$.
The subgroups $A\times1$ and $1\times B$ commute in $L$.  Hence the
\emph{mixed commutator pairing}
\begin{equation}
\label{eqn: mixed commutator pairing}
 \beta_\mu:A\times B\longrightarrow k^\times,
 \qquad
 \beta_\mu(a,b)
 =\Alt_\mu\bigl((a,1),(1,b)\bigr)
 =\frac{\mu((a,1),(1,b))}{\mu((1,b),(a,1))}
\end{equation}
is defined.  It depends only on $[\mu]$ and is multiplicative in both
variables.  Its simultaneous conjugation invariance is
\[
 \beta_\mu(gag^{-1},hbh^{-1})=\beta_\mu(a,b)
 \quad ((g,h)\in L,\ a\in A,\ b\in B).
\]

The following criterion was obtained by Davydov in terms of ribbon
equivalences of centers \cite[Corollary~3.6.3]{D}.  For equal groups, a
bimodule-category proof and the inverse construction appear in
\cite[Proposition~5.2 and Remark~5.5]{NR}.
Galindo--Plavnik \cite[Theorem~7.9 and Remark~7.10]{GP} give the criterion
for bimodule categories between twisted pointed categories.
We recover the untwisted criterion from the Clifford composition formula.

\begin{theorem}
\label{thm: invertibility criterion}
The transitive decorated $(G,H)$-biset $[L,\mu]$ is invertible if and only if
\begin{equation}
\label{eqn: invertibility conditions}
 p_1(L)=G,
 \qquad
 p_2(L)=H,
 \qquad
 \beta_\mu:A\times B\longrightarrow k^\times
 \text{ is a perfect pairing}.
\end{equation}
Equivalently, $L$ is a subdirect product, $A$ and $B$ are abelian, and
\eqref{eqn: mixed commutator pairing} is nondegenerate.

Define $(g,h)^\dagger=(h^{-1},g^{-1})$ and
$L^\dagger=\{\ell^\dagger\mid \ell\in L\}\leq H\times G$,
and let $\mu^\dagger$ be the cocycle on $L^\dagger$ given by
$\mu^\dagger(\ell^\dagger,m^\dagger) =\mu(\ell^{-1},m^{-1})$ for
$\ell,m\in L$.
When \eqref{eqn: invertibility conditions} holds,
\begin{equation}
\label{eqn: inverse general decorated biset}
 [L,\mu]^{-1}=[L^\dagger,(\mu^\dagger)^{-1}].
\end{equation}
\end{theorem}

\begin{proof}
Let $[M,\nu]$ be the inverse, with normalized cocycles $\mu,\nu$.
Output stabilizers in the product over $H$ lie in
$p_1(L)\times p_2(M)$, so its being the identity forces
$p_1(L)=p_2(M)=G$.  The reverse product gives
$p_2(L)=p_1(M)=H$.

Put $C=k_1(M)$, $D=k_2(M)$, $K=B\cap C$, and $N=L*M$.
On $F_1=L\times_HM$ put
$\eta_1=\operatorname{pr}_L^*\mu\,\operatorname{pr}_M^*\nu$
and $\alpha=\eta_1|_K$.  There is one double coset and one
$N$-orbit in $\Irr(k_\alpha[K])$, with stabilizer $I$ conjugate to
$\Delta(G)$.  Since $A\times1$ and $1\times D$ lie in $N$ and meet $I$
trivially,
\[
 \max\{|A|,|D|\}\leq[N:I]
 =|\Irr(k_\alpha[K])|\leq|K|\leq\min\{|B|,|C|\}.
\]
The reverse product similarly gives
$\max\{|B|,|C|\}\leq|D\cap A|\leq\min\{|A|,|D|\}$.
Thus $A=D$, $B=C=K$, $|A|=|B|$, and
$|\Irr(k_\alpha[B])|=|B|$.  All simple modules are one-dimensional,
so $B$ is abelian.  Their stabilizers are conjugate to $\Delta(G)$;
hence $A\times1$ acts freely and, by equality of cardinalities,
simply transitively.

For $a\in A$ and $b\in B=C$, set
$\widetilde a=((a,1),(1,1))$ and
$\widetilde b=((1,b),(b,1))$ in $F_1$.
These elements commute.  The corresponding units of $k_{\eta_1}[F_1]$ satisfy
\begin{equation}
\label{eqn: mixed commutator Clifford action}
 u_{\widetilde a}u_{\widetilde b}
 =\beta_\mu(a,b)\,u_{\widetilde b}u_{\widetilde a}.
\end{equation}
The contribution of $\nu$ is $1$, since
$\nu((1,1),(b,1))=\nu((b,1),(1,1))=1$.
Thus the homogeneous component containing $u_{\widetilde a}$ sends
a simple module $\rho$ to $\rho\otimes\beta_\mu(a,-)^{-1}$.
Freeness makes $A\to\widehat B$, $a\mapsto\beta_\mu(a,-)$, injective;
equality of orders makes it an isomorphism, proving perfectness.

Conversely, assume \eqref{eqn: invertibility conditions} and compose
$[L,\mu]$ with
$[L^\dagger,(\mu^\dagger)^{-1}]$.  There is one double coset, and
in triple coordinates
\[
 \begin{aligned}
 F_1&=\{(a,b,c):(a,b),(c,b)\in L\},\\
 \eta_1\bigl((a,b,c),(a',b',c')\bigr)
 &=\frac{\mu((a,b),(a',b'))}{\mu((c,b),(c',b'))}.
 \end{aligned}
\]
The restriction to the kernel $B$ is trivial, so its simple modules are
indexed by $\widehat B$.  Perfectness and the calculation above make the
$A\times1$-action simply transitive.
Also $\eta_1=1$ on
$q_1^{-1}(\Delta(G))=\{(a,b,a):(a,b)\in L\}$.
The trivial character of $B$ extends to the trivial representation of this
preimage.  Thus $\Delta(G)$ fixes it, with intertwiners satisfying the
cocycle equation with scalar $1$.
Since $N_1=\{(a,c):aA=cA\}$ has order $|G||A|$ and the orbit has size
$|A|$, the stabilizer is exactly $\Delta(G)$.
The Clifford formula gives $[\Delta(G),1]$; interchanging the endpoints
gives $[\Delta(H),1]$ for the reverse product.
\end{proof}

Goursat's lemma gives the following reformulation; see also
\cite[Lemma~3.6.4]{D}.
Let $A\triangleleft G$ and $B\triangleleft H$ be normal subgroups and let
$\phi:G/A\xrightarrow{\ \sim\ }H/B$.
The corresponding subdirect product is
\begin{equation}
\label{eqn: Goursat subdirect product}
 L_{A,B,\phi}
 =\{(g,h)\in G\times H\mid \phi(gA)=hB\}.
\end{equation}
If $A$ and $B$ are abelian, let $Q=G/A$ act on $A$ by conjugation
and on $B$ through $\phi:Q\to H/B$.  These actions are independent
of the chosen lifts.  With trivial $Q$-action on $k^\times$, the
simultaneous conjugation invariance of $\beta_\mu$ gives
\begin{equation}
\label{eqn: cross commutator map}
 \operatorname{cr}_{L_{A,B,\phi}}:
 H^2(L_{A,B,\phi},k^\times)
 \longrightarrow
 \Hom_Q(A\otimes B,k^\times),
 \qquad [\mu]\longmapsto\beta_\mu.
\end{equation}

\begin{corollary}
\label{cor: Goursat cohomological Morita criterion}
The groups $G$ and $H$ are categorically Morita equivalent if and only if
there exist abelian normal subgroups $A\triangleleft G$ and
$B\triangleleft H$, an isomorphism
$\phi:G/A\xrightarrow{\sim}H/B$, and a class
$[\mu]\in H^2(L_{A,B,\phi},k^\times)$ whose mixed commutator is perfect.
Thus the Morita equivalences between $\Vec_G$ and $\Vec_H$ are represented
by the Goursat data $(A,B,\phi,[\mu])$
with this property, modulo $G\times H$-conjugacy of the associated pairs
$(L_{A,B,\phi},[\mu])$.
\end{corollary}

A perfect pairing identifies $B$ with $\widehat A$ as a $Q$-module by
$b\mapsto\beta_\mu(-,b)$, where
$(q\cdot\chi)(a)=\chi(q^{-1}\cdot a)$.
The extensions $1\to A\to G\to Q\to1$ and $1\to B\to H\to Q\to1$,
with the second quotient map $h\mapsto\phi^{-1}(hB)$, have fiber product
$G\times_QH=L_{A,B,\phi}$.
For pointed Morita equivalence see also \cite{NaiduMorita}; an
extension-theoretic formulation is given in
\cite[Theorem~3.9]{UribeMorita}.


\section{Extremal structure and the canonical one-sided action}
\label{sect: top and bottom ideals}

Put $R_G=\DBurn(G,G)$, and let $\mathcal B_G$ be its distinguished
basis, with invertible part $\mathcal B_G^\times$.
The noninvertible basis elements span the unique largest proper based ideal
$I_G$.  At the other extreme, the basis elements that factor through the
trivial group span the \emph{bottom ideal} $J_G$, the unique minimal nonzero
based ideal.  The composition matrix $C_G$ controls both the bottom ideal and
the restriction of the canonical one-sided action to it; its categorical
interpretation will be given below.  We compute the radical of
$\mathbb C\otimes_{\mathbb Z}J_G$ and prove that this algebra is semisimple
precisely when $G$ is cyclic.  We also prove that the full canonical
one-sided action is faithful only for the trivial group.

\subsection{Duality and the top quotient}
\label{subsect: singular ideal}

Bimodule duality gives a basis-preserving anti-involution
$(-)^\vee:\DBurn(G,H)\to\DBurn(H,G)$, satisfying
$(x\circ_Hy)^\vee=y^\vee\circ_Hx^\vee$.
The dual of a bimodule category $\mathcal M$ is its opposite category
$\mathcal M^\vee=\mathcal M^{\mathrm{op}}$, with the two actions interchanged
and twisted by categorical duals; see \cite[Section~2.9]{ENO2}.  Let
$\tau_G:R_G\longrightarrow\mathbb Z$
be the coefficient of the identity $[\Delta(G),1]$, and put
\begin{equation}
\label{eq: coefficient identity pairing}
 \langle x,y\rangle_G:=\tau_G(xy^\vee).
\end{equation}
The degeneracy of this form detects precisely the noninvertible part of the
basis.  We first record a consequence of the Clifford formula.

\begin{proposition}
\label{prop: invertible constituent}
Let $x,y$ be basis elements of $R_G$.  If an invertible basis element occurs
in $xy$, then both $x$ and $y$ are invertible.
\end{proposition}

\begin{proof}
Multiplication by an invertible basis element permutes the distinguished
basis.  After multiplying on the left by the inverse of the given
constituent, we may therefore assume that the identity occurs in $xy$.
Write $x=[L,\mu]$ and $y=[M,\nu]$, with normalized cocycles, and choose
a double-coset representative giving an identity summand in the Clifford
formula.  After conjugating $(M,\nu)$ we may take this representative
to be $1$.

Put $A=k_1(L)$, $B=k_2(L)$, $C=k_1(M)$, and $D=k_2(M)$.
For the corresponding Clifford data put $K=B\cap C$, $N=L*M$, and
$\alpha=\alpha_1$.  Choose $\pi\in\Irr(k_\alpha[K])$ giving the
identity summand and write $I=N_{1,\pi}$.  Then $I$ is conjugate to
$\Delta(G)$, so its coordinate projections are surjective and its
coordinate kernels are trivial.  Since
$I\leq N\leq p_1(L)\times p_2(M)$, we obtain
$p_1(L)=p_2(M)=G$.  Moreover, $A\times1$ and $1\times D$ lie in $N$
and both meet $I$ trivially.

Goursat's lemma gives $|B|\leq|A|$ and $|C|\leq|D|$.
Since $I$ is the stabilizer of $[\pi]$,
\[
 \max\{|A|,|D|\}
 \leq[N:I]
 \leq|\Irr(k_\alpha[K])|
 \leq|K|
 \leq\min\{|B|,|C|\}
 \leq\min\{|A|,|D|\}.
\]
Thus equality holds throughout.  In particular, $K=B=C$ and
$|A|=|B|=|C|=|D|$.  Goursat's lemma now also gives
$p_2(L)=p_1(M)=G$.

Since $|\Irr(k_\alpha[K])|=|K|$, every simple
$k_\alpha[K]$-module is one-dimensional.  Thus the twisted group
algebra is commutative, and $K=B=C$ is abelian.
For $a\in A$ and $b\in B=C$, take the lifts
$\widetilde a=((a,1),(1,1))$ and
$\widetilde b=((1,b),(b,1))$ in $F_1=L\times_G M$.
On the transported module
$(\widetilde A_1)_{(a,1)}\ot_{A_1}\pi$,
\eqref{eqn: mixed commutator Clifford action} gives
\[
 u_{\widetilde b}(u_{\widetilde a}\ot v)
 =\beta_\mu(a,b)^{-1}
   u_{\widetilde a}\ot\pi(u_{\widetilde b})v.
\]
Thus $(a,1)$ sends $\pi$ to $\pi\ot\beta_\mu(a,-)^{-1}$.
If $\beta_\mu(a,-)=1$, it fixes $[\pi]$, so $(a,1)\in I$ and
$a=1$.  Hence $A\to\widehat B$, $a\mapsto\beta_\mu(a,-)$, is
injective.  Equality of orders makes it an isomorphism, so $A$ is
abelian and $\beta_\mu$ is perfect.  The same stabilizer argument
using $1\times D$ proves that the mixed commutator pairing for
$[M,\nu]$ is perfect.  Both factors are therefore invertible by
Theorem~\ref{thm: invertibility criterion}.
\end{proof}

\begin{corollary}
\label{cor: singular ideal}
The subgroup
\begin{equation}
\label{eq: singular ideal definition}
 I_G:=\mathbb Z\text{-span}
 \bigl(\mathcal B_G\setminus\mathcal B_G^\times\bigr)
\end{equation}
is a two-sided based ideal.  It is the unique largest proper based ideal of
$R_G$, and
\begin{equation}
\label{eq: top quotient BrPic}
 R_G/I_G\cong\mathbb Z[\BrPic(\Vec_G)].
\end{equation}
Moreover, $\tau_G(xy)=\tau_G(yx)$, and the radical of
\eqref{eq: coefficient identity pairing} is exactly $I_G$.
\end{corollary}

\begin{proof}
Proposition~\ref{prop: invertible constituent} shows that $I_G$ is a
two-sided based ideal and gives \eqref{eq: top quotient BrPic}.
Every proper based ideal contains no invertible basis element, hence is
contained in $I_G$.

For basis elements $b,c$, Proposition~\ref{prop: invertible constituent}
also gives
\[
 \langle b,c\rangle_G=
 \begin{cases}
  1,&b=c\in\mathcal B_G^\times,\\
  0,&\text{otherwise}.
 \end{cases}
\]
The assertions about the trace and the radical follow by linearity.
\end{proof}

\subsection{The bottom ideal and its Lagrangian Gram matrix}
\label{subsect: bottom sandwich ideal}

Recall that $\Omega_G$ is the distinguished basis of the one-sided decorated Burnside ring
$\OneBurn(G)=\DBurn(G,1) \cong K_0\bigl(\mathbf{Mod}(\Vec_G)\bigr)$.
Thus $\Omega_G$ consists of the conjugacy classes of pairs
$a=[H,\alpha]$, where $H\leq G$ and
$[\alpha]\in H^2(H,k^\times)$.  For $a,b\in\Omega_G$ define
\begin{equation}
\label{eq: bottom matrix units}
 E_{a,b}:=a\circ_1b^\vee\in R_G.
\end{equation}
These are distinct basis elements.  In subgroup--cocycle coordinates,
$E_{[H,\alpha],[K,\beta]}$ has product subgroup $H\times K$; restricting
its decoration to the two coordinate subgroups recovers $[\alpha]$ and the
dual of $[\beta]$.  If $\mathcal M_a$ denotes the indecomposable
$\Vec_G$-module category represented by $a$, then under $\Phi_{G,G}$,
$E_{a,b}$ corresponds to $\mathcal M_a\boxtimes\mathcal M_b^\vee$.

There are positive integers $c_{b,c}$ defined by
\begin{equation}
\label{eq: one sided composition pairing}
 b^\vee\circ_Gc=c_{b,c},
 \qquad b,c\in\Omega_G.
\end{equation}
The matrix $C_G=(c_{b,c})$ is symmetric.  More explicitly, for
$b=[H,\alpha]$ and $c=[K,\beta]$,
\begin{equation}
\label{eq: one sided composition matrix entries}
 c_{b,c}
 =\sum_{g\in H\backslash G/K}
 \left|\Irr
 \left(k_{\theta_g}[H\cap{}^gK]\right)\right|,
 \qquad
 [\theta_g]
 =\left[\alpha^{-1}|_{H\cap{}^gK}\,
 {}^g\!\beta|_{H\cap{}^gK}\right].
\end{equation}
This is a specialization of the Clifford formula; compare
\cite[Proposition~4.10]{EtingofKinserWalton}.  Symmetry also follows
directly by taking duals in \eqref{eq: one sided composition pairing}.

Let
\begin{equation}
\label{eq: bottom ideal definition}
 J_G:=\mathbb Z\text{-span}
 \{E_{a,b}\mid a,b\in\Omega_G\}.
\end{equation}

\begin{proposition}
\label{prop: bottom sandwich ideal}
The subgroup $J_G$ is a two-sided based ideal of $R_G$, and
\begin{equation}
\label{eq: bottom sandwich multiplication}
 E_{a,b}E_{c,d}=c_{b,c}E_{a,d}.
\end{equation}
It is the unique minimal nonzero based ideal of $R_G$.
\end{proposition}

\begin{proof}
Associativity of Clifford composition and
\eqref{eq: one sided composition pairing} give
\[
 (a\circ_1b^\vee)(c\circ_1d^\vee)
 =a\circ_1(b^\vee\circ_Gc)\circ_1d^\vee
 =c_{b,c}E_{a,d}.
\]
More generally, composing on either side with an arbitrary element of
$R_G$ still factors through the trivial group, so $J_G$ is a two-sided
based ideal.

Let $K$ be a nonzero based ideal and choose a basis element $x\in K$.
For any $c\in\Omega_G$, the one-sided product $x\circ_Gc$ is nonzero; choose
$b\in\Omega_G$ occurring in it.  Then
$b^\vee\circ_Gx\circ_Gc=n$ for some $n>0$.  Consequently, for every
$a,d\in\Omega_G$,
$E_{a,b}\,x\,E_{c,d}=nE_{a,d}$.
Because $K$ is based, it contains $E_{a,d}$ for all $a,d$.  Hence
$J_G\subseteq K$.
\end{proof}

The composition matrix appearing in the bottom ideal has a categorical
analogue for every fusion category.  We use \emph{Lagrangian algebra} in its
standard sense:
a connected étale algebra of maximal Frobenius--Perron dimension in a
nondegenerate braided fusion category.

\begin{theorem}
\label{thm: Lagrangian Gram matrix}
Let $\mathcal C$ be a fusion category, let
$\{\mathcal M_i\}_{i\in I}$ represent its indecomposable module categories,
and let $A_i\in\Z(\mathcal C)$ be the Lagrangian algebra associated with
$\mathcal M_i$.  Then
\begin{equation}
\label{eq: Lagrangian Gram identity}
 \operatorname{rk}\operatorname{Fun}_{\mathcal C}
       (\mathcal M_i,\mathcal M_j)
 =\dim\Hom_{\Z(\mathcal C)}(A_i,A_j).
\end{equation}
If $X$ ranges over the simple objects of $\Z(\mathcal C)$, put
$m_{X,i}:=\dim\Hom_{\Z(\mathcal C)}(X,A_i)$ and $M=(m_{X,i})_{X,i}$.
Then the matrix on the left-hand side of
\eqref{eq: Lagrangian Gram identity} is
\begin{equation}
\label{eq: Lagrangian Gram factorization}
 M^tM.
\end{equation}
In particular, it is positive semidefinite, and it is nonsingular if and
only if the classes $[A_i]$ are linearly independent in
$K_0(\Z(\mathcal C))_{\mathbb Q}$.
\end{theorem}

\begin{proof}
Fix $i,j$ and put
\[
 \mathcal C_{\mathcal M_i}^*
 :=\operatorname{Fun}_{\mathcal C}(\mathcal M_i,\mathcal M_i),
 \qquad
 \mathcal D:=(\mathcal C_{\mathcal M_i}^*)^{\mathrm{rev}},
 \qquad
 \mathcal P
 :=\operatorname{Fun}_{\mathcal C}(\mathcal M_i,\mathcal M_j).
\]
Morita transport gives a braided equivalence
$\Z(\mathcal C)\simeq\Z(\mathcal D)$ that carries $A_i$ to the canonical
Lagrangian algebra $A_{\mathcal D}$ and $A_j$ to the Lagrangian algebra
$A_{\mathcal P}$ associated with the $\mathcal D$-module category
$\mathcal P$; see \cite[Section~2.9]{ENO2} and
\cite[Proposition~4.8]{DMNO}.  Let
$F_{\mathcal D}:\Z(\mathcal D)\to\mathcal D$ be the forgetful functor.
As an object of $\Z(\mathcal D)$, $A_{\mathcal D}$ is also the image of
$\mathbf1_{\mathcal D}$ under the left adjoint of $F_{\mathcal D}$.
Moreover, by \cite[Corollary~3.15]{ShimizuCoends},
\begin{equation}
\label{eq: underlying Lagrangian algebra}
 F_{\mathcal D}(A_{\mathcal P})
 \cong\int_{p\in\mathcal P}\uHom_{\mathcal D}(p,p)
 \cong
 \bigoplus_{p\in\Irr(\mathcal P)}
 \uHom_{\mathcal D}(p,p),
\end{equation}
where $\uHom_{\mathcal D}$ denotes the internal Hom.
Using the left adjunction for $A_{\mathcal D}$ and then the defining
property of the internal Hom gives
\begin{align*}
 \dim\Hom_{\Z(\mathcal D)}(A_{\mathcal D},A_{\mathcal P})
 &=
 \sum_{p\in\Irr(\mathcal P)}
 \dim\Hom_{\mathcal D}
       \bigl(\mathbf 1_{\mathcal D},\uHom_{\mathcal D}(p,p)\bigr)\\
 &=
 \sum_{p\in\Irr(\mathcal P)}
 \dim\End_{\mathcal P}(p)
 =\operatorname{rk}(\mathcal P).
\end{align*}
This proves \eqref{eq: Lagrangian Gram identity}.

Since $\Z(\mathcal C)$ is semisimple, decomposing the two algebras into
simple objects gives
$\dim\Hom_{\Z(\mathcal C)}(A_i,A_j) =\sum_Xm_{X,i}m_{X,j}$.
This is \eqref{eq: Lagrangian Gram factorization}.  Finally,
$M^tM$ and $M$ have the same kernel over $\mathbb Q$, which proves the last
assertion.
\end{proof}

For $\mathcal C=\Vec_G$, let $A_a\in\Z(\Vec_G)$ denote the Lagrangian
algebra associated with $a\in\Omega_G$.  The relative tensor-product
equivalence
$\mathcal M_b^\vee\boxtimes_{\Vec_G}\mathcal M_c \simeq\operatorname{Fun}_{\Vec_G}(\mathcal M_b,\mathcal M_c)$
from \cite[Proposition~3.5]{ENO2} identifies
\begin{equation}
\label{eq: sandwich coefficients Lagrangian Hom}
 c_{b,c}
 =\operatorname{rk}\operatorname{Fun}_{\Vec_G}
       (\mathcal M_b,\mathcal M_c)
 =\dim\Hom_{\Z(\Vec_G)}(A_b,A_c).
\end{equation}
Thus $C_G$ is the Gram matrix of the underlying object classes $[A_a]$.
In particular,
\begin{equation}
\label{eq: sandwich kernel Lagrangian relations}
 \ker(C_G)
 =
 \left\{(\lambda_a)_{a\in\Omega_G}\ \middle|\
 \sum_{a\in\Omega_G}\lambda_a[A_a]=0
 \text{ in }K_0(\Z(\Vec_G))_{\mathbb C}\right\}.
\end{equation}

\begin{remark}
In the classical Burnside setting, the integral kernel of the Gram
matrix of permutation characters consists precisely of the Brauer
relations among the corresponding permutation representations.
Here the permutation representations are replaced by the underlying
classes of the Lagrangian algebras $A_a$.  It would be interesting to study
these Lagrangian relations in their own right.
\end{remark}

We now pass to the complexified bottom ideal.  Put
$V_G:=\mathbb C\Omega_G$, $\mathcal J_G:=\mathbb C\otimes_{\mathbb Z}J_G$,
and $n_G:=|\Omega_G|$.
The matrix $C_G$ is the matrix of the symmetric bilinear form
$c_G$ on $V_G$ defined by
$c_G(b,c)=c_{b,c}$.  Under the vector-space identification
$\mathcal J_G\cong V_G\otimes V_G$, with $E_{a,b}\longleftrightarrow a\otimes b$,
the product is contraction of the two middle factors by $c_G$:
$(a\otimes b)(c\otimes d)=c_G(b,c)a\otimes d$.
Equivalently, after ordering $\Omega_G$ and identifying $E_{a,b}$ with a
matrix unit,
\begin{equation}
\label{eq: sandwich matrix product}
 X*Y=XC_GY.
\end{equation}
The vector space of matrices with this multiplication is called a
\emph{sandwich algebra}.

For the canonical representation $\rho_G$ on $\mathbb Z\Omega_G$
of Corollary~\ref{thm: one-sided Clifford action}, given by
$\rho_G(x)(c)=x\circ_Gc$, one has, for $a,b,c\in\Omega_G$,
\begin{equation}
\label{eq: bottom ideal one sided action}
\begin{aligned}
 \rho_G(E_{a,b})(c)
 &=E_{a,b}\circ_Gc
 =a\circ_1(b^\vee\circ_Gc)
 =c_{b,c}a,\\
 \rho_G(X)&=XC_G.
\end{aligned}
\end{equation}
In the second formula we use the same notation for the complexified
action on $V_G$ and take $X\in\mathcal J_G$.
Thus each basis element $E_{a,b}\in J_G$ acts on $\mathbb Z\Omega_G$
by the rank-one operator $c\mapsto c_G(b,c)a$.
The action of $J_G$ is therefore determined by the bilinear form $c_G$,
which may be degenerate.

\subsection{Non-faithfulness of the canonical one-sided action}
\label{subsect: nonfaithful one-sided action}

In addition to $V_G=\mathbb C\Omega_G$ and
$\mathcal J_G=\mathbb C\otimes_{\mathbb Z}J_G$, put
\[
 V_{G,\mathbb Q}:=\mathbb Q\Omega_G,
 \qquad
 J_{G,\mathbb Q}:=\mathbb Q\otimes_{\mathbb Z}J_G.
\]
Under the matrix identification above, the rationalized bottom-ideal
action is
\[
 J_{G,\mathbb Q}\longrightarrow
 \operatorname{End}_{\mathbb Q}(V_{G,\mathbb Q}),
 \qquad
 X\longmapsto XC_G.
\]

\begin{proposition}
\label{prop: one-sided action nonfaithful}
The canonical representation
$\rho_G:R_G\to\operatorname{End}_{\mathbb Z}(\OneBurn(G))$
is faithful if and only if $G=1$.  Suppose $G\neq1$.  If $C_G$ is
singular, then $\ker(\rho_G)\cap J_G\neq0$.  If $C_G$ is nonsingular,
then $X\mapsto XC_G$ gives an isomorphism
\[
 J_{G,\mathbb Q}\xrightarrow{\ \sim\ }
 \operatorname{End}_{\mathbb Q}(V_{G,\mathbb Q}).
\]
In the latter case, let $d$ be a positive integer such that
$dC_G^{-1}$ is integral, and let $X_d\in J_G$ correspond to
$dC_G^{-1}$.  Then $d\,\mathbf1_G-X_d$ is a nonzero element of
$\ker(\rho_G)$.
\end{proposition}

\begin{proof}
For $G=1$, both $R_G$ and $\Omega_k^{(2)}(G)$ are canonically
$\mathbb Z$, and the action is ordinary multiplication, hence faithful.

Assume $G\neq1$.  By the bottom-ideal action formula
\eqref{eq: bottom ideal one sided action}, the restriction of $\rho_G$ to
$J_G$, after ordering the basis $\Omega_G$, is the matrix map
$X\mapsto XC_G$.

If $C_G$ is singular, there is a nonzero rational matrix $X$ with
$XC_G=0$.  Clearing denominators gives a nonzero integral matrix, still
denoted $X$, and hence a nonzero element of $J_G$ annihilating
$\Omega_k^{(2)}(G)$.  Thus $\ker(\rho_G)\cap J_G\neq0$.

Suppose now that $C_G$ is nonsingular.  Then right multiplication by
$C_G$ is an isomorphism
\[
 J_{G,\mathbb Q}\xrightarrow{\ \sim\ }
 \operatorname{End}_{\mathbb Q}(V_{G,\mathbb Q}).
\]
Choose $d>0$ such that $dC_G^{-1}$ is integral and let $X_d\in J_G$ be the
corresponding matrix element.  Since
\[
 \rho_G(X_d)=d\,\operatorname{id}_{V_G}
 =\rho_G(d\,\mathbf1_G),
\]
the element $d\,\mathbf1_G-X_d$ lies in $\ker(\rho_G)$.

It is nonzero.  Indeed, for $G\neq1$, the bottom ideal satisfies
$J_G\subseteq I_G$, while $\mathbf1_G$ is an invertible distinguished
basis element and hence does not belong to $I_G$.  Therefore
$d\,\mathbf1_G\notin J_G$, so it cannot equal $X_d$.  This proves the
proposition.
\end{proof}

\begin{example}
With the notation $\mathbf1,\omega,q,\ell,r,e$ of
Example~\ref{ex: decorated double Burnside C2}, the action table in
Example~\ref{ex: one-sided decorated Burnside C2 action} gives
\[
 \ker(\rho_{C_2})
 =
 \mathbb Z\bigl(\ell+r-\mathbf1-2\omega\bigr)
 \oplus
 \mathbb Z\bigl(q+e-2\mathbf1-\omega\bigr).
\]
Indeed, applying either displayed relation to $\xi$ and to $\tau$ gives
zero directly from that table.  Row reduction of the resulting four by six
integer action matrix shows that these two relations generate the full
kernel.
\end{example}

\subsection{The radical and nonsemisimplicity of the bottom ideal}
\label{subsect: radical bottom ideal}

\begin{proposition}
\label{prop: radical sandwich algebra}
Let $r_G=\operatorname{rank}_{\mathbb C}(C_G)$ and
$K_G=\operatorname{Rad}(c_G)$.  Then
\begin{equation}
\label{eq: radical sandwich tensor}
 \operatorname{Jac}(\mathcal J_G)
 =K_G\otimes V_G+V_G\otimes K_G
 =\{X\in\operatorname{Mat}_{n_G}(\mathbb C)\mid C_GXC_G=0\}.
\end{equation}
Moreover,
\begin{equation}
\label{eq: sandwich semisimple quotient}
 \mathcal J_G/\operatorname{Jac}(\mathcal J_G)
 \cong\operatorname{Mat}_{r_G}(\mathbb C),
 \qquad
 \dim\operatorname{Jac}(\mathcal J_G)=n_G^2-r_G^2.
\end{equation}
\end{proposition}

\begin{proof}
Pass from $V_G$ to the quotient
$\overline V_G=V_G/K_G$.  The form induced by $c_G$ on
$\overline V_G$ is nondegenerate, and contraction identifies
\[
 \overline V_G\otimes\overline V_G
 \cong\operatorname{End}(\overline V_G)
 \cong\operatorname{Mat}_{r_G}(\mathbb C).
\]
The kernel of the resulting algebra homomorphism from $\mathcal J_G$ is
$K_G\otimes V_G+V_G\otimes K_G$.  Its square is contained in
$K_G\otimes K_G$, which annihilates $\mathcal J_G$ on both sides; hence its
cube is zero.  The kernel is therefore the Jacobson radical, and the
dimension formula follows.  Reducing $C_G$ by invertible row and column
operations to $\operatorname{diag}(I_{r_G},0)$ gives the equivalent matrix
description \eqref{eq: radical sandwich tensor}.
\end{proof}

Since $\mathcal J_G$ is an ideal in
$\mathcal A_G:=\mathbb C\otimes_{\mathbb Z}R_G$, the hereditary property of
the Jacobson radical for finite-dimensional algebras gives
\begin{equation}
\label{eq: bottom radical in full radical}
 \operatorname{Jac}(\mathcal J_G)
 =\mathcal J_G\cap\operatorname{Jac}(\mathcal A_G).
\end{equation}
Hence a singular $C_G$ implies that the full complexified ring is not
semisimple.

\begin{lemma}
\label{lem: untwisted mark factorization}
Let $\Omega_G^0\subseteq\Omega_G$ consist of the elements $[H,1]$, and
let $C_G^0$ be the corresponding principal submatrix of $C_G$.
Let $\mathcal A_2(G)$ represent the conjugacy classes of abelian
subgroups of $G$ generated by at most two elements.  Put
$m_A(H)=|(G/H)^A|$ and
$M=(m_A(H))_{A\in\mathcal A_2(G),\,[H,1]\in\Omega_G^0}$.
Let $D=\operatorname{diag}(d_A)_{A\in\mathcal A_2(G)}$, where
\[
 d_A=
 \frac{|\{(x,y)\in A^2:\langle x,y\rangle=A\}|}{|N_G(A)|}.
\]
Then
\begin{equation}
\label{eq: untwisted mark factorization}
 C_G^0=M^tDM.
\end{equation}
Consequently,
\begin{equation}
\label{eq: untwisted sandwich rank}
 \operatorname{rank}(C_G^0)=|\mathcal A_2(G)|,
 \qquad
 \dim\ker(C_G^0)=|\Omega_G^0|-|\mathcal A_2(G)|.
\end{equation}
In particular, $C_G^0$ is nonsingular if and only if $G$ is abelian and
generated by at most two elements.
\end{lemma}

\begin{proof}
For subgroups $H,K\leq G$, write $c_{H,K}=c_{[H,1],[K,1]}$.
The untwisted specialization of
\eqref{eq: one sided composition matrix entries} is
$c_{H,K}=\sum_{g\in H\backslash G/K}k(H\cap{}^gK)$,
where $k(L)$ is the number of conjugacy classes of $L$.
We first show that
\begin{equation}
\label{eq: commuting pair mark formula}
 c_{H,K}
 =\frac1{|G|}
 \sum_{\substack{x,y\in G\\xy=yx}}
 |(G/H)^{\langle x,y\rangle}|\,
 |(G/K)^{\langle x,y\rangle}|.
\end{equation}
The product of the two fixed-point numbers counts fixed points on
$G/H\times G/K$.  For a transitive orbit $G/L$, write $G_z$ for the
stabilizer of $z\in G/L$; thus $G_{gL}=gLg^{-1}$.  Its contribution is
\[
\begin{aligned}
 \frac1{|G|}
 \sum_{\substack{x,y\in G\\xy=yx}}
 |(G/L)^{\langle x,y\rangle}|
 &=\frac1{|G|}
   \sum_{z\in G/L}
   \bigl|\{(x,y)\in G_z^2:xy=yx\}\bigr|\\
 &=\frac{|G:L|}{|G|}
   \bigl|\{(x,y)\in L^2:xy=yx\}\bigr|
 =k(L).
\end{aligned}
\]
The diagonal $G$-orbits of $G/H\times G/K$ have stabilizers
$H\cap{}^gK$, indexed by $H\backslash G/K$.  Summing their
contributions gives \eqref{eq: commuting pair mark formula}.

Group the sum by the conjugacy class of $A=\langle x,y\rangle$.
For $A\in\mathcal A_2(G)$ there are $[G:N_G(A)]$ conjugates of $A$,
each with the same number of ordered generating pairs.  Since the
fixed-point numbers are conjugacy-invariant, their total contribution is
\[
 \frac{[G:N_G(A)]}{|G|}
 \bigl|\{(x,y)\in A^2:\langle x,y\rangle=A\}\bigr|\,
 m_A(H)m_A(K)
 =d_A\,m_A(H)m_A(K).
\]
Summing over $A\in\mathcal A_2(G)$ gives
\eqref{eq: untwisted mark factorization}.  Every $d_A$ is positive.

To see that $M$ has full row rank, restrict its columns to
$H\in\mathcal A_2(G)$ and order both rows and columns by nondecreasing
subgroup order.  One has $m_A(H)\neq0$ precisely when
$g^{-1}Ag\leq H$ for some $g\in G$.  This requires $|A|\leq|H|$;
if their orders are equal, it requires $A$ and $H$ to be conjugate,
hence the same chosen representative.  The resulting square matrix is
therefore upper triangular, with diagonal entries
$m_A(A)=|N_G(A):A|>0$.  Hence $M$ has full row rank, and the positivity
of $D$ gives \eqref{eq: untwisted sandwich rank}.

Since $\Omega_G^0$ is indexed by all conjugacy classes of subgroups,
$C_G^0$ is nonsingular exactly when every subgroup of $G$ is abelian
and generated by at most two elements.  Necessity follows by taking
$G$ itself.  Conversely, every subgroup of a finite abelian group
generated by at most two elements is again generated by at most two
elements.
\end{proof}

\begin{proposition}
\label{prop: elementary abelian Lagrangian relation}
Let $p$ be a prime and $E=C_p\times C_p$.  Order $\Omega_E$ by the two
families
\begin{align*}
 \Omega_E^+
 &:=\{[1,1]\}\cup
 \{[E,\alpha]\mid[\alpha]\in H^2(E,k^\times)\},\\
 \Omega_E^-
 &:=\{[L,1]\mid L\leq E,\ |L|=p\}.
\end{align*}
Each family has $p+1$ elements.  If $J_{p+1}$ is the all-ones matrix of
size $p+1$, then
\begin{equation}
\label{eq: elementary abelian sandwich matrix}
 C_E=
 \begin{pmatrix}
  (p^2-1)I_{p+1}+J_{p+1}&pJ_{p+1}\\
  pJ_{p+1}&(p^2-1)I_{p+1}+J_{p+1}
 \end{pmatrix}.
\end{equation}
The matrix $C_E$ has rank $2p+1$, hence is singular of corank one.
The associated Lagrangian algebras satisfy
\begin{equation}
\label{eq: elementary abelian Lagrangian relation}
 [A_{[1,1]}]
 +\sum_{[\alpha]\in H^2(E,k^\times)}[A_{[E,\alpha]}]
 =
 \sum_{\substack{L\leq E\\|L|=p}}[A_{[L,1]}]
\end{equation}
in $K_0(\Z(\Vec_E))$.
\end{proposition}

\begin{proof}
Identify the Lagrangian algebras in $\Z(\Vec_E)$ with the Lagrangian
subgroups of $\mathbb H(E)=E\oplus\widehat E$, with
$q(x,\chi)=\chi(x)$.
See Proposition~\ref{prop: Omega Lagrangian correspondence} for this
correspondence.
The family $\Omega_E^+$ corresponds to
$0\oplus\widehat E$ and the graphs of
$x\mapsto\Alt_\alpha(x,-)^{-1}$; the family $\Omega_E^-$ corresponds
to $L\oplus L^\perp$ for lines $L\leq E$, where
$L^\perp=\{\chi\in\widehat E:\chi|_L=1\}$.

Two distinct subgroups in either family intersect only in zero.
For the graphs, the ratio of two distinct alternating bicharacters
on $E$ is nondegenerate; for the other family, distinct lines and
their annihilators have trivial intersections.  Subgroups in opposite
families meet in a line: on a graph this follows from
$\Alt_\alpha|_{L\times L}=1$, while
$(0\oplus\widehat E)\cap(L\oplus L^\perp)=0\oplus L^\perp$.
Each subgroup has $p^2$ elements.  The Gram entries are the orders of
these intersections, giving
\eqref{eq: elementary abelian sandwich matrix}.

Every nonzero isotropic element $(x,\chi)$ belongs to exactly one
subgroup in each family.  For $x\neq0$, the second has
$L=\langle x\rangle$, and evaluation at $x$ bijects the alternating
bicharacters on $E$ with the characters trivial on $\langle x\rangle$,
which determines the graph in the first family.  For $x=0$ and
$\chi\neq1$, the two subgroups are $0\oplus\widehat E$ and
$\ker\chi\oplus(\ker\chi)^\perp$.  The zero element belongs to all
$p+1$ subgroups in each family.  Hence the sums of the corresponding
Lagrangian algebra classes are equal, proving
\eqref{eq: elementary abelian Lagrangian relation}.

The eigenvalues of $C_E$ are $p^2-1$, $2p(p+1)$, and $0$, with
multiplicities $2p$, $1$, and $1$, respectively.  Thus its rank is
$2p+1$, and the displayed relation is the only one up to scalar.
\end{proof}

\begin{theorem}
\label{thm: bottom cyclicity criterion}
For a finite group $G$, the following are equivalent:
\begin{enumerate}
\item $G$ is cyclic;
\item the Gram matrix $C_G$ is nonsingular;
\item the complex bottom ideal
      $\mathcal J_G=\mathbb C\otimes_{\mathbb Z}J_G$ is semisimple.
\end{enumerate}
\end{theorem}

\begin{proof}
The equivalence of \textup{(ii)} and \textup{(iii)} follows from
Proposition~\ref{prop: radical sandwich algebra}.  If $G$ is cyclic, then
every subgroup is cyclic and $H^2(H,k^\times)=0$ for every $H\leq G$.
Thus $\Omega_G=\Omega_G^0$, and
Lemma~\ref{lem: untwisted mark factorization} shows that $C_G=C_G^0$ is
nonsingular.  Hence \textup{(i)} implies \textup{(ii)}, and therefore
\textup{(iii)}.

Suppose that $G$ is nonabelian.  By
Lemma~\ref{lem: untwisted mark factorization}, the principal submatrix
$C_G^0$ is singular.  If $C_G$ were nonsingular, then, being a Gram
matrix, it would be positive definite.  But every principal submatrix
of a positive definite matrix is positive definite, a contradiction.

It remains to consider an abelian noncyclic group.  Choose a prime $p$ and
a quotient $\pi:G\twoheadrightarrow E=C_p\times C_p$, and put
$K=\ker(\pi)$.  Inside the hyperbolic metric group
$\mathbb H(G)=G\oplus\widehat G$, the subgroup $S=K\oplus0$ is isotropic
and $S^\perp/S\cong E\oplus\widehat E=\mathbb H(E)$.
Taking inverse images under $S^\perp\to S^\perp/S$ sends distinct
Lagrangian subspaces of $\mathbb H(E)$ to distinct Lagrangian
subgroups of $\mathbb H(G)$.  An element of $S^\perp$ lies in such
an inverse image exactly when its image lies in the original subgroup.
Thus \eqref{eq: elementary abelian Lagrangian relation} lifts to a
nonzero relation among the underlying classes of Lagrangian algebras
in $\Z(\Vec_G)$.
By the correspondence used above, each of these algebras is $A_a$ for
some $a\in\Omega_G$.  Hence
\eqref{eq: sandwich kernel Lagrangian relations} makes $C_G$ singular.
\end{proof}

The following two examples illustrate degeneracy caused by cocycle
decorations and degeneracy already present in the untwisted matrix.

\begin{example}
\label{ex: Klein four bottom degeneracy}
For $p=2$, Proposition~\ref{prop: elementary abelian Lagrangian relation}
gives the first example with singular $C_G$, namely $V=C_2\times C_2$.
In the ordering
$[1,1],\ [H_1,1],\ [H_2,1],\ [H_3,1],\ [V,1],\ [V,\beta]$,
where $H_1,H_2,H_3$ are the subgroups of order $2$ and $\beta$ is the
nontrivial class in $H^2(V,k^\times)$, the matrix $C_V$ is the one
computed in \cite[Example~4.20]{EtingofKinserWalton} in terms of the
ranks of the corresponding bimodule categories.
Its kernel is generated by $(1,-1,-1,-1,1,1)^t$.
Thus $r_V=5$ and
$\dim\operatorname{Jac}(\mathcal J_V)=6^2-5^2=11$.
By Lemma~\ref{lem: untwisted mark factorization}, $C_V^0$ is
nonsingular, whereas $C_V$ is singular.  This is the smallest example
in which a cocycle decoration is essential to the dependence.
\end{example}

\begin{example}
Let $G=S_3$.  Put
$H_2=\langle(12)\rangle$ and $H_3=\langle(123)\rangle$.
All Schur multipliers relevant to the one-sided basis are trivial, so order
that basis as $[1,1]$, $[H_2,1]$, $[H_3,1]$, $[S_3,1]$.
Then
\begin{equation}
\label{eq: S3 sandwich matrix}
 C_{S_3}=
 \begin{pmatrix}
  6&3&2&1\\
  3&3&1&2\\
  2&1&6&3\\
  1&2&3&3
 \end{pmatrix}.
\end{equation}
This matrix has rank $3$ and kernel generated by $(1,-2,-1,2)^t$.
Hence $\dim\operatorname{Jac}(\mathcal J_{S_3})=16-9=7$, and
$\mathbb C\otimes R_{S_3}$ is not semisimple.  Here no nontrivial
one-sided cocycle decoration is available: the singularity comes from the
nonabelian subgroup geometry alone.
\end{example}


\section{Factorization and Mackey structure}
\label{sect: factorization and Mackey}

Having determined the extremal structure, we turn to factorization
through smaller groups.  Standard two-sided subgroup induction satisfies
a Mackey decomposition whose double-coset summands factor through
subgroup intersections.

\subsection{Factorization through smaller groups}
\label{subsect: factorization middle}

For $G\neq1$, let $F_{<G}$ be the based span of all constituents of
composites that factor through a strictly smaller group:
\begin{equation}
\label{eq: smaller group factorization ideal}
 F_{<G}:=\sum_{|H|<|G|}
 \DBurn(G,H)\circ_H\DBurn(H,G)\subseteq R_G.
\end{equation}
The sum denotes the span of the distinguished basis constituents of these
composites.

For a subgroup $L\leq G$, let $F_G(L)$ denote the based span of the
indecomposable constituents of all composites
\[
 \mathcal X\boxtimes_{\Vec_L}\mathcal Y,
\]
where $\mathcal X$ is a $(\Vec_G,\Vec_L)$-bimodule category and
$\mathcal Y$ is a $(\Vec_L,\Vec_G)$-bimodule category.  Equivalently,
$F_G(L)$ is the based span of the constituents of
\[
 \DBurn(G,L)\circ_L\DBurn(L,G).
\]
A basis element lies in $F_G(L)$ precisely when it occurs as a
constituent of a composite through $L$.  If $L<G$,
then $F_G(L)\subseteq F_{<G}$.

\begin{proposition}
\label{prop: top bottom filtration}
For every nontrivial finite group $G$ there is a canonical chain of
two-sided based ideals
\begin{equation}
\label{eq: top bottom filtration}
 0\subset J_G\subseteq F_{<G}\subseteq I_G\subset R_G.
\end{equation}
Every basis element $[L,\mu]$ for which $p_1(L)<G$ or $p_2(L)<G$ belongs
to $F_{<G}$.  Consequently $I_G/F_{<G}$ is spanned by the images of
noninvertible subdirect pairs $(L,[\mu])$ with $p_1(L)=p_2(L)=G$.
\end{proposition}

\begin{proof}
Associativity shows that the span of morphisms factoring through a fixed
group $H$ is stable under multiplication on the left and right by $R_G$;
taking all basis constituents and summing over $|H|<|G|$ gives a based
two-sided ideal.  The ideal $J_G$ consists of the factorizations through
$1$, so $J_G\subseteq F_{<G}$.

Suppose that $[L,\mu]\circ_H[M,\nu]$, with
$L\leq G\times H$, $M\leq H\times G$, and $|H|<|G|$, has an
invertible constituent.  Multiplying on the left by its inverse,
we may assume that this constituent is the identity.  Choose the
double-coset representative producing that summand and, after
conjugating $(M,\nu)$, take it to be $1$.
Put $A=k_1(L)$, $B=k_2(L)$, and write
$K=B\cap k_1(M)$, $N=L*M$, and $\alpha=\alpha_1$ for the
Clifford data.  Choose $\pi\in\Irr(k_\alpha[K])$ producing the
identity and let $I$ be its stabilizer in $N$.
Then $I$ is conjugate to $\Delta(G)$, so
$p_1(L)=p_2(M)=G$.  Moreover, $A\times1\leq N$ and
$(A\times1)\cap I=1$.  Orbit--stabilizer therefore gives
\[
 |A|\leq[N:I]\leq|\Irr(k_\alpha[K])|\leq|K|\leq|B|.
\]
On the other hand, Goursat's lemma gives
\[
 \frac{|G|}{|A|}
 =\frac{|p_2(L)|}{|B|}
 \leq\frac{|H|}{|B|},
\]
so $|B|<|A|$, a contradiction.  Thus
$F_{<G}\subseteq I_G$.

Finally, suppose $P=p_1(L)<G$.  Regard $(L,[\mu])$ first as a decorated
$(P,G)$-biset.  Composing it on the left with the undecorated inclusion
biset represented by $\Delta(P)\leq G\times P$ reproduces
$[L,\mu]$ and factors it through $P$.  The argument using
$Q=p_2(L)<G$ is symmetric.  Thus only subdirect pairs can remain in the
essential quotient.
\end{proof}

We call $I_G/J_G$ the \emph{middle quotient}.
The filtration \eqref{eq: top bottom filtration} separates it into the
subquotients $F_{<G}/J_G$ and $I_G/F_{<G}$.
The former is inherited from factorizations through smaller groups;
the latter is the essential quotient.
In a Clifford product, passing from the image $N_g$ to the stabilizer
$N_{g,\pi}$ can shrink the coordinate projections.

\begin{remark}
\label{rem: middle quotients}
The structure of $I_G/F_{<G}$ and of $I_G/J_G$ is open in general.
Subsection~\ref{subsect: cyclic structural example} describes the
filtration by factorization through subgroups for cyclic $p$-groups
and proves semisimplicity of the complexified ring for cyclic groups.
\end{remark}

\subsection{Standard subgroup induction and Mackey decomposition}
\label{subsect: standard subgroup induction Mackey}

For $H\leq G$ and a finite semisimple $\Vec_H$-bimodule category
$\mathcal M$, define its
\emph{standard two-sided subgroup induction} by
\begin{equation}
\label{eq: standard two-sided subgroup induction}
 \operatorname{Ind}_H^G(\mathcal M)
 :=\Vec_G\boxtimes_{\Vec_H}\mathcal M
          \boxtimes_{\Vec_H}\Vec_G,
\end{equation}
where the two copies of $\Vec_G$ have their
$(\Vec_G,\Vec_H)$- and $(\Vec_H,\Vec_G)$-bimodule structures
induced by the inclusion $H\leq G$.
On decorated-biset classes we use the same notation:
\[
 \operatorname{Ind}_H^G(x)
 :=[\Delta(H),1]_{G,H}\circ_H x\circ_H
   [\Delta(H),1]_{H,G},
 \qquad x\in R_H.
\]
For $H,K\leq G$ and $D\in H\backslash G/K$, let $\Vec_D$ be the
full subcategory of $\Vec_G$ whose simple objects are indexed by $D$,
with its $(\Vec_H,\Vec_K)$-bimodule structure inherited from
multiplication in $G$.

The same double-coset decomposition underlies the classical Mackey
formula.  The $(H,K)$-biset $G$ is the disjoint union of the $HgK$, and
each such biset factors through $g^{-1}Hg\cap K$, with its embedding into
$H$ given by conjugation by $g$.  Replacing these bisets by the
corresponding bimodule categories gives the following decomposition of
products of standard inductions.

\begin{theorem}
\label{thm: standard subgroup Mackey}
Let $H,K\leq G$, and let $\mathcal M$ and $\mathcal N$ be finite
semisimple $\Vec_H$- and $\Vec_K$-bimodule categories, respectively.
There is an equivalence of $\Vec_G$-bimodule categories
\begin{equation}
\label{eq: Mackey decomposition}
 \operatorname{Ind}_H^G(\mathcal M)
 \boxtimes_{\Vec_G}\operatorname{Ind}_K^G(\mathcal N)
 \simeq\bigoplus_{D\in H\backslash G/K}
   \Vec_G\boxtimes_{\Vec_H}\mathcal M
   \boxtimes_{\Vec_H}\Vec_D
   \boxtimes_{\Vec_K}\mathcal N
   \boxtimes_{\Vec_K}\Vec_G.
\end{equation}
For $g\in D$, the $D$-summand factors through
$\Vec_{g^{-1}Hg\cap K}$.  Consequently,
\[
 [\operatorname{Ind}_H^G(\mathcal M)]
 [\operatorname{Ind}_K^G(\mathcal N)]
 \in\sum_{D=HgK\in H\backslash G/K}F_G(g^{-1}Hg\cap K).
\]
\end{theorem}

\begin{proof}
Associativity of relative tensor product identifies the left-hand side
of \eqref{eq: Mackey decomposition} with
\[
 \Vec_G\boxtimes_{\Vec_H}\mathcal M
 \boxtimes_{\Vec_H}\Vec_G
 \boxtimes_{\Vec_K}\mathcal N
 \boxtimes_{\Vec_K}\Vec_G.
\]
As a $(\Vec_H,\Vec_K)$-bimodule category, the middle copy of
$\Vec_G$ decomposes as
\begin{equation}
\label{eq: Mackey double coset decomposition}
 \Vec_G\simeq\bigoplus_{D\in H\backslash G/K}\Vec_D.
\end{equation}
Distributing relative tensor product over this finite direct sum gives
\eqref{eq: Mackey decomposition}.

Fix $g\in D$ and put $I_g=g^{-1}Hg\cap K$.  Give $H$ the right
$I_g$-action $u\cdot t=ugtg^{-1}$, and give $K$ the left
$I_g$-action by multiplication.  The map
\[
 H\times_{I_g}K\xrightarrow{\ \sim\ }HgK,
 \qquad [u,v]\longmapsto ugv,
\]
is an $(H,K)$-biset isomorphism: the balancing relation
$(ugtg^{-1},v)\sim(u,tv)$ is respected because
$(ugtg^{-1})gv=ugtv$, and the map is surjective between sets
of order $|H||K|/|I_g|$.
The balancing action of $I_g$ on $H\times K$ is free.  With the
trivial decorations, the biset isomorphism therefore gives
\begin{equation}
\label{eq: Mackey graph factorization}
 \Vec_D\simeq\Vec_H\boxtimes_{\Vec_{I_g}}\Vec_K,
\end{equation}
where the right $\Vec_{I_g}$-action on $\Vec_H$ is induced by
$t\mapsto gtg^{-1}$.  Substituting this factorization into the
$D$-summand proves that it factors through $\Vec_{I_g}$, and hence
that its class belongs to $F_G(I_g)$.  The summands themselves are
independent of the representative, since they are defined from
$\Vec_D$.
\end{proof}

The conclusion concerns factorization through $g^{-1}Hg\cap K$; an individual
summand need not be a standard two-sided induction from a
$\Vec_{g^{-1}Hg\cap K}$-bimodule category.

Under the identification of
Corollary~\ref{cor: decorated bisets and bimodule categories},
\eqref{eq: Mackey decomposition} gives the corresponding identity
in $R_G$, which extends to arbitrary $x\in R_H$ and $y\in R_K$
by bilinearity.  The Clifford formula of
Theorem~\ref{thm: categorical Clifford composition} expands each
double-coset summand into indecomposable classes.

\begin{remark}
Let
\[
 d_H:=\operatorname{Ind}_H^G(\mathbf1_H)=[\Delta(H),1].
\]
Then the Mackey formula specializes to
\[
 d_Hd_K
 =
 \sum_{HgK\in H\backslash G/K}d_{g^{-1}Hg\cap K}.
\]
For $H=K=1$, this gives $d_1^2=|G|d_1$.
\end{remark}

\subsection{Cyclic groups}
\label{subsect: cyclic structural example}

For $G=C_p$, the noninvertible basis elements are precisely the four
undecorated classes $[H\times K,1]$, with $H,K\in\{1,G\}$, so
$I_G=J_G$.  Indeed, cyclic subgroups have trivial second cohomology, the
noncoordinate lines are graphs of automorphisms, and every nontrivial
decoration on $G\times G$ has perfect mixed commutator; the remaining
classes are therefore invertible by
Theorem~\ref{thm: invertibility criterion}.  We now determine the
complexified ring for arbitrary cyclic groups.  For
$N=\prod_{p\mid N}p^{n_p}$, external product gives
$R_{C_N}\cong\bigotimes_{p\mid N}R_{C_{p^{n_p}}}$ as based rings.  Indeed,
every subgroup of $C_N\times C_N$ and its second cohomology class decompose
uniquely over the primes, and the Clifford formula computes products
componentwise.  It therefore suffices to consider cyclic groups of
prime-power order.  We identify the factorization ideals by a numerical
height and prove semisimplicity by computing the successive quotients.
Fix a prime $p$ and write
$G_j=\mathbb Z/p^j\mathbb Z$, with $G_0=1$ as a group, using additive
coordinates.  For $j\leq m$, identify $G_j$ with the standard subgroup
of $G_m$ by $x\mapsto p^{m-j}x$; standard quotient maps are reduction
maps.  For $[L,\mu]\in\DBurn(G_m,G_n)$, write
$p_1(L)=G_a$ and $p_2(L)=G_b$.  By Goursat's lemma, let $G_s$ be the
common quotient
$G_a/k_1(L)\cong G_b/k_2(L)\cong G_s$.  If
$\operatorname{ord}[\mu]=p^t$, with $t=0$ for the trivial class, define
the \emph{factorization height} of $[L,\mu]$ by
\begin{equation}
\label{eq: cyclic factorization height}
 \operatorname{ht}(L,\mu):=s+t.
\end{equation}

\begin{proposition}
\label{prop: cyclic factorization height}
Let $[L,\mu]\in\DBurn(G_m,G_n)$ be as above.
\begin{enumerate}
\item[\textup{(i)}] One has $s+t\leq\min(a,b)$, and $[L,\mu]$
factors through $G_{s+t}$.

\item[\textup{(ii)}] If the monomial expression supplied by
Corollary~\ref{cor: abelian Clifford formula} is
$[L,\mu]\circ_{G_n}[M,\nu]=c[N,\lambda]$ with $c>0$, then
$\operatorname{ht}(N,\lambda)\leq
\min\{\operatorname{ht}(L,\mu),\operatorname{ht}(M,\nu)\}$.
\end{enumerate}
\end{proposition}

\begin{proof}
For \textup{(i)}, interchange the coordinates if necessary so that
$a\geq b$, and choose compatible generators.  For a unit $u$ modulo
$p^s$, lifted modulo $p^b$, put $e=(1,u)$ and $f=(0,p^s)$.  Then
$L=\{(x,y)\in G_a\times G_b:y\equiv ux\pmod{p^s}\}$ and
$L=\langle e\rangle\oplus\langle f\rangle\cong G_a\oplus G_{b-s}$.
Since $k^\times$ is divisible,
$H^2(L,k^\times)\cong\Hom(\wedge^2L,k^\times)\cong C_{p^{b-s}}$.
For $\beta=\Alt_\mu$, the class is determined by $\beta(e,f)$, whose
order is $p^t$.  Hence $t\leq b-s$, proving
$s+t\leq b=\min(a,b)$.

Put $h=s+t$, let $q_a:G_a\twoheadrightarrow G_h$ and
$q_b:G_b\twoheadrightarrow G_h$ be the standard quotient maps, and set
$\overline L=(q_a\times q_b)(L)$.  The kernel of
$L\to\overline L$ is $D=\langle p^he,p^tf\rangle$.  To show that
$D\subseteq\Rad(\beta)$, it suffices to pair its generators with $e$ and
$f$.  The pairings of multiples of $e$ with $e$, and of multiples of
$f$ with $f$, are trivial by alternation, while
\[
 \beta(p^he,f)=\beta(e,f)^{p^h}=1,
 \qquad
 \beta(p^tf,e)=\beta(f,e)^{p^t}=1.
\]
Here $\beta(e,f)$ has order $p^t$ and $h\geq t$.  Therefore
$D\subseteq\Rad(\beta)$, and $\beta$ descends to an alternating
bicharacter $\overline\beta$ on $\overline L$; for its class
$[\overline\mu]$, one has
$(q_a\times q_b)^*[\overline\mu]=[\mu]$.

Put
\[
 Q_a=\left[\{(u,q_a(u)):u\in G_a\},1\right]
     \in\DBurn(G_m,G_h),
 \qquad
 Q_b=\left[\{(v,q_b(v)):v\in G_b\},1\right]
     \in\DBurn(G_n,G_h).
\]
Here $G_a\leq G_m$ and $G_b\leq G_n$ are understood by the standard
subgroup convention above.  The exact factorization through $G_h$ is
\begin{equation}
\label{eq: cyclic exact factorization}
 [L,\mu]
 =
 Q_a\circ_{G_h}[\overline L,\overline\mu]
 \circ_{G_h}Q_b^\vee.
\end{equation}
Write $\Gamma(q_a)$ for the graph of $q_a$.
For the first composition, $p_2(\Gamma(q_a))=G_h$ because $q_a$ is
surjective, so Corollary~\ref{cor: abelian Clifford formula} has a single
double coset.  Also $k_2(\Gamma(q_a))=1$, so its intersection with
$k_1(\overline L)$ is trivial; the twisted group algebra on this
intersection is $k$, with one simple module and hence one Clifford orbit.
The composition therefore gives one term, with the pullback decoration,
appearing once.  On the right, the reversed graph of $q_b$ has full first
projection $G_h$ and trivial first-coordinate kernel, so the same
conclusions apply.  The resulting subgroup is the inverse image of
$\overline L$ under $q_a\times q_b$, namely $L$, and its decoration is
$[\mu]$ by the pullback identity above.  This proves
\eqref{eq: cyclic exact factorization} and hence factorization through
$G_h$.

For \textup{(ii)}, put $x=[L,\mu]\in\DBurn(G_m,G_n)$ and
$y=[M,\nu]\in\DBurn(G_n,G_\ell)$, and suppose
$x\circ_{G_n}y=c[N,\lambda]$ with $c>0$.  Set
$h=\operatorname{ht}(L,\mu)$ and use
\eqref{eq: cyclic exact factorization}.

We first record that these graph factors preserve height.  If
$[S,\sigma]\in\DBurn(G_h,G_d)$, then
$Q_a\circ_{G_h}[S,\sigma]=[S',\sigma']$, where
\[
 S'=\{(u,v)\in G_a\times G_d:(q_a(u),v)\in S\}
\]
is regarded as a subgroup of $G_m\times G_d$, and
$[\sigma']=(q_a\times\id)^*[\sigma]$.
One has $p_1(S')=q_a^{-1}(p_1(S))$ and
$k_1(S')=q_a^{-1}(k_1(S))$, and therefore
$p_1(S')/k_1(S')\cong p_1(S)/k_1(S)$.
Moreover, the map $S'\to S$, given by
$(u,v)\mapsto(q_a(u),v)$, is surjective and
$\Alt_{\sigma'}=(q_a\times\id)^*\Alt_\sigma$.
The two alternating bicharacters therefore have the same image in
$k^\times$, so their cohomology classes have the same order and
$\operatorname{ht}(S',\sigma')=\operatorname{ht}(S,\sigma)$.
The analogous right-hand conclusion for composition with $Q_b^\vee$
follows by duality.

Applying Corollary~\ref{cor: abelian Clifford formula} successively gives
\[
 [\overline L,\overline\mu]\circ_{G_h}Q_b^\vee
 \circ_{G_n}y=c'[Z,\eta],
 \qquad c'>0,
\]
where $[Z,\eta]\in\DBurn(G_h,G_\ell)$.  Part \textup{(i)} gives
$\operatorname{ht}(Z,\eta)\leq h$.  By associativity,
$x\circ_{G_n}y=c'Q_a\circ_{G_h}[Z,\eta]$, and height preservation for
$Q_a$ gives
$\operatorname{ht}(N,\lambda)\leq h=\operatorname{ht}(L,\mu)$.
Apply the same conclusion to the dual product
$(x\circ_{G_n}y)^\vee=y^\vee\circ_{G_n}x^\vee$.
Duality preserves height: interchanging coordinates and
inverting and transporting the cocycle changes neither the order of the
Goursat quotient nor the order of the decoration.  Therefore
$\operatorname{ht}(N,\lambda)\leq\operatorname{ht}(M,\nu)$, proving
\textup{(ii)}.
\end{proof}

\begin{corollary}
\label{cor: cyclic factorization filtration}
For $G_n=C_{p^n}$ and $0\leq r\leq n$,
\begin{equation}
\label{eq: cyclic factorization ideal by height}
 F_{G_n}(G_r)
 =
 \mathbb Z\text{-span}
 \{[L,\mu]\in R_{G_n}:\operatorname{ht}(L,\mu)\leq r\}.
\end{equation}
These ideals form the factorization filtration
\begin{equation}
\label{eq: cyclic factorization chain}
 J_{G_n}=F_{G_n}(G_0)
 \subseteq F_{G_n}(G_1)
 \subseteq\cdots\subseteq
 F_{G_n}(G_n)=R_{G_n}.
\end{equation}
Moreover, $\operatorname{ht}(L,\mu)$ is the least $r$ for which
$[L,\mu]$ factors through $G_r$.
\end{corollary}

\begin{proof}
If a basis element occurs in a composite through
$G_r$, Proposition~\ref{prop: cyclic factorization height}\textup{(i)}
bounds the heights of both factors by $r$, and part \textup{(ii)} bounds
the output by $r$.  Hence $F_{G_n}(G_r)$ is
contained in the right-hand side of
\eqref{eq: cyclic factorization ideal by height}.

Conversely, let $[L,\mu]$ have height $h\leq r$.  The factorization
\eqref{eq: cyclic exact factorization} is through $G_h$.  The undecorated
restriction--inclusion pair for $G_h\leq G_r$ satisfies
\[
 [\Delta(G_h),1]_{G_h,G_r}\circ_{G_r}
 [\Delta(G_h),1]_{G_r,G_h}=p^{r-h}1_{G_h}.
\]
Inserting this pair into \eqref{eq: cyclic exact factorization} expresses
a positive multiple of $[L,\mu]$ as a composite through $G_r$.  Since
$F_{G_n}(G_r)$ is the based span of the constituents of such composites,
$[L,\mu]\in F_{G_n}(G_r)$.  This proves the reverse inclusion.
No division in the integral ring is involved.

The height-span description in
\eqref{eq: cyclic factorization ideal by height} is increasing in $r$.
Its endpoints are $F_{G_n}(G_0)=F_{G_n}(1)=J_{G_n}$ and $R_{G_n}$, the
latter because all heights are at most $n$; this proves
\eqref{eq: cyclic factorization chain}.  If $[L,\mu]$ has height $h$,
\eqref{eq: cyclic exact factorization} supplies a factorization through
$G_h$, while any factorization through $G_r$ forces $h\leq r$ by the
height bound and monotonicity in
Proposition~\ref{prop: cyclic factorization height}\textup{(i),(ii)}.
Thus $h$ is the least possible index.
\end{proof}

Fix $n$ and write
$\mathcal F_r:=\mathbb C\otimes_{\mathbb Z}F_{G_n}(G_r)$ for
$0\leq r\leq n$.  For every $r\geq1$, put
\begin{equation}
\label{eq: cyclic layer algebra definition}
 E_r
 :=
 \mathbb C\otimes_{\mathbb Z}
 \left(R_{G_r}/F_{G_r}(G_{r-1})\right).
\end{equation}

\begin{proposition}
\label{prop: cyclic matrix layers}
For $1\leq r\leq n$, there is an algebra isomorphism
\begin{equation}
\label{eq: cyclic matrix layer}
 \mathcal F_r/\mathcal F_{r-1}
 \cong
 \operatorname{Mat}_{n-r+1}(E_r).
\end{equation}
\end{proposition}

\begin{proof}
For $0\leq i\leq n-r$, let
$q_i:G_{r+i}\twoheadrightarrow G_r$ be the standard quotient with kernel
$G_i$, chosen compatibly so that
$q_k|_{G_{r+j}}=p^{k-j}q_j$ for $j\leq k$.  Define
\[
 \ell_i=\left[\{(x,q_i(x)):x\in G_{r+i}\},1\right]
 \in\DBurn(G_n,G_r),
\]
where $G_{r+i}$ is viewed as the standard subgroup of $G_n$.
Its dual $\ell_i^\vee\in\DBurn(G_r,G_n)$ is the reversed graph.
Every height-$r$ basis element satisfies
$p_1(L)=G_{r+i}$ and $p_2(L)=G_{r+j}$ for unique
$0\leq i,j\leq n-r$, and descends uniquely along $q_i\times q_j$ to a
height-$r$ basis element $x$ of $R_{G_r}$.  The decoration descends
because $G_i\times1$ and $1\times G_j$ lie in $L$ and in the radical of
$\Alt_\mu$.  Thus
\[
 [L,\mu]=\ell_i\circ_{G_r}x\circ_{G_r}\ell_j^\vee.
\]
Conversely, every such expression has height $r$.  The identities
$p_1(L)=G_{r+i}$ and $p_2(L)=G_{r+j}$ determine $i,j$, while the image
of $L$ under $q_i\times q_j$, equipped with the unique cohomology class
whose pullback is $[\mu]$, recovers $x$.

By Corollary~\ref{cor: abelian Clifford formula} and the definition of
$\ell_i$, for $0\leq j\leq k\leq n-r$ one has
\begin{equation}
\label{eq: cyclic quotient section product}
 \ell_j^\vee\circ_{G_n}\ell_k
 =p^{n-r-|j-k|}
 \left[\{(a,p^{k-j}a):a\in G_r\},1\right].
\end{equation}
For $k<j$, interchange the coordinates in the graph on the right.

For $j=k$, the output subgroup is $\Delta(G_r)$.  If $j\neq k$, the
undecorated output has height
$\max\{r-|j-k|,0\}\leq r-1$.  Consequently
\begin{equation}
\label{eq: cyclic quotient section diagonal}
 \ell_j^\vee\ell_k
 \equiv
 \delta_{jk}p^{n-r}1_{G_r}
 \pmod{F_{G_r}(G_{r-1})}.
\end{equation}

Let $e_{ij}$ be the standard matrix units, and write
$e_{ij}(\overline x)$ for the matrix having
$\overline x\in E_r$ in entry $(i,j)$ and zero elsewhere.  Define
\[
 \Phi:\operatorname{Mat}_{n-r+1}(E_r)
 \longrightarrow
 \mathcal F_r/\mathcal F_{r-1},
 \qquad
 \Phi\bigl(e_{ij}(\overline x)\bigr)
 =p^{-(n-r)}\,\overline{\ell_i x\ell_j^\vee}.
\]
Here $x\in\mathbb C\otimes_{\mathbb Z}R_{G_r}$ is any lift of
$\overline x\in E_r$, and composition is extended
complex-bilinearly.  Height monotonicity from
Proposition~\ref{prop: cyclic factorization height} makes the map
independent of the lift, since
Corollary~\ref{cor: cyclic factorization filtration} places lower-height
terms in $\mathcal F_{r-1}$.
By \eqref{eq: cyclic quotient section diagonal},
\[
 \begin{aligned}
 \Phi(e_{ij}(\overline x))\Phi(e_{kl}(\overline y))
 &=p^{-2(n-r)}
   \overline{\ell_i x\ell_j^\vee\ell_k y\ell_l^\vee}\\
 &=\delta_{jk}p^{-(n-r)}
   \overline{\ell_i(xy)\ell_l^\vee}\\
 &=\Phi\bigl(e_{ij}(\overline x)e_{kl}(\overline y)\bigr).
 \end{aligned}
\]
The normal form above shows that these images form a basis.  Thus
$\Phi$ is a linear bijection and hence the isomorphism in
\eqref{eq: cyclic matrix layer}.
\end{proof}

For $r=0$, the bottom must be treated separately.  Here
$\mathcal F_0=\mathbb C\otimes_{\mathbb Z}J_{G_n}$.  In the order
$G_0,G_1,\ldots,G_n$, its Gram matrix $C_{G_n}$ is
$\bigl(p^{n-|i-j|}\bigr)_{0\leq i,j\leq n}$.  It is nonsingular by
Theorem~\ref{thm: bottom cyclicity criterion}, and hence
\begin{equation}
\label{eq: cyclic bottom matrix algebra}
 \mathcal F_0
 \cong
 \operatorname{Mat}_{n+1}(\mathbb C).
\end{equation}

\begin{proposition}
\label{prop: cyclic positive layers semisimple}
For every $r\geq1$, the algebra $E_r$ is semisimple over $\mathbb C$.
\end{proposition}

\begin{proof}
Fix $r\geq1$, put $G=G_r$, choose a primitive $p^r$-th root of unity
$\zeta\in k^\times$, and write
$\mathsf U_j=(\mathbb Z/p^j\mathbb Z)^\times$, with
$\mathsf U_0=\{1\}$.  For $0\leq s\leq r$,
$u\in\mathsf U_s$, and $c\in\mathsf U_{r-s}$, let $X_{s,u,c}$ be
$p^{-\min(s,r-s)}$ times the image in $E_r$ of the decorated basis
element supported on
\[
 \{(x,y)\in G^2:y\equiv ux\pmod{p^s}\},
\]
with alternating commutator
\[
 ((x,y),(x',y'))\longmapsto\zeta^{c(xy'-x'y)}.
\]
This is well defined because $xy'-x'y$ is divisible by $p^s$, while
$c$ is taken modulo $p^{r-s}$.  Proposition~\ref{prop: cyclic factorization height}
and Corollary~\ref{cor: cyclic factorization filtration} show that these
are precisely the surviving classes with full coordinate projections
and height $r$, so the $X_{s,u,c}$ form a basis of $E_r$.  Here $s$
records the order $p^s$ of the common quotient, $u$ its isomorphism, and
$c$ the decoration of order $p^{r-s}$.  Let $S_r$ consist of these
normalized basis elements together with the absorbing zero.  The
normalization is chosen so that every product is zero or another
normalized basis element, as proved below.

For $X_{s,u,c}$ and
$X_{t,v,d}$, put
\[
 h=\min\{s,t,r-s,r-t\},
 \qquad
 q=\max\{\min(s,t),\,r-\max(s,t)\}.
\]
The two congruences below overlap modulo $p^h$; when compatible, they
determine the output slope modulo $p^q$.
The product is zero unless $vu\equiv d^{-1}c\pmod{p^h}$.
When this congruence holds, let $w\in\mathsf U_q$ be the unique unit
satisfying
\[
 w\equiv vu\pmod{p^{\min(s,t)}},
 \qquad
 w\equiv d^{-1}c\pmod{p^{r-\max(s,t)}}.
\]
Then
\[
 X_{s,u,c}X_{t,v,d}
 =
 \begin{cases}
  X_{q,w,cv^{-1}},&s\leq t,\\[1mm]
  X_{q,w,du},&s\geq t,
 \end{cases}
\]
where the final parameter is reduced modulo $p^{r-q}$.  When $s=t$,
the two expressions agree by the compatibility congruence.  Congruences
at exponent zero have their evident meaning, with $\mathsf U_0=\{1\}$,
and the formula also holds for $p=2$.  The parameter $q$ is the output
common-quotient exponent, not its factorization height; every surviving
output has height $r$.

We derive this formula from
Corollary~\ref{cor: abelian Clifford formula}.  Both middle projections
are $G$, so there is a single double coset.  For the underlying
unnormalized decorated elements, the fiber product and middle kernel are
\[
 \begin{aligned}
 F&=\{(x,y,z):
   y\equiv ux\pmod{p^s},\
   z\equiv vy\pmod{p^t}\},\\
 K&=p^{\max(s,t)}G,
 &|K|&=p^{r-\max(s,t)}.
 \end{aligned}
\]
Here $K$ is embedded in $F$ as the elements $(0,k,0)$.
The alternating commutator of the product decoration is
\[
 \begin{aligned}
 B\bigl((x,y,z),(x',y',z')\bigr)
 &=\zeta^{c(xy'-x'y)+d(yz'-y'z)},\\
 B((x,y,z),k)&=\zeta^{k(cx-dz)}\quad(k\in K).
 \end{aligned}
\]
Therefore the stabilizer, in coordinates $(x,z)$, is
\[
 N=
 \left\{(x,z):
  z\equiv vux\pmod{p^{\min(s,t)}},\
  dz\equiv cx\pmod{p^{r-\max(s,t)}}
 \right\}.
\]
Its first projection is full exactly when
$vu\equiv d^{-1}c\pmod{p^h}$.  If compatibility fails, the output has
proper first projection and hence height at most $r-1$; it vanishes in
$E_r$.

Assume compatibility.  The first congruence defining $N$ comes from the
underlying composition and the second from the Clifford stabilizer;
together they determine the support indexed by $q,w$ above.  The image of $F$ under
$(x,y,z)\mapsto(x,z)$ is defined by the first congruence alone.
This image acts on the simple modules of the twisted algebra on $K$
by the character twists determined by the displayed pairing, with
stabilizer $N$.  Each orbit therefore has size
$p^{q-\min(s,t)}$.
Since $K$ is cyclic, the twisted algebra has
$p^{r-\max(s,t)}$ simple modules.
Thus the number of orbits is
\[
 p^{r-\max(s,t)-(q-\min(s,t))}=p^h.
\]
On the inverse image of $N$ in $F$ under
$(x,y,z)\mapsto(x,z)$, the kernel $K$ lies in the radical of $B$, so the
restricted bicharacter descends to $N$.  If $s\leq t$, use $y=v^{-1}z$
and obtain
$B=\zeta^{cv^{-1}(xz'-x'z)}$; if $s\geq t$, use $y=ux$ and obtain
$B=\zeta^{du(xz'-x'z)}$.
These computations give the two final parameters above.
The unnormalized Clifford coefficient is $p^h$, and it disappears after
normalization because
$\min(s,r-s)+\min(t,r-t)=h+\min(q,r-q)$.
This proves the normalized multiplication formula and shows that $S_r$
is closed under multiplication.

Recall that a semigroup is inverse if each element $x$ has a unique
$x^*$ satisfying $xx^*x=x$ and $x^*xx^*=x^*$
\cite[Section~2.1]{SteinbergMobiusII}.  To prove that $S_r$ is inverse,
it suffices to exhibit such inverses and to show that all idempotents
commute \cite[Proposition~8.1]{SteinbergMobiusII}.

Define $0^*=0$ and $X_{s,u,c}^*=X_{s,u^{-1},c}$.  For
$a=\min(s,r-s)$, direct substitution
gives
\[
 X_{s,u,c}X_{s,u,c}^*
 =X_{r-a,1,cu},
 \qquad
 X_{s,u,c}^*X_{s,u,c}
 =X_{r-a,1,cu^{-1}}.
\]
The final parameters are reduced modulo $p^a$.  Hence $XX^*X=X$ and
$X^*XX^*=X^*$.  The nonzero idempotents are exactly
$X_{s,1,c}$ with $2s\geq r$.  Indeed, idempotence forces
$\max(s,r-s)=s$ and $u^2=u$, hence $2s\geq r$ and $u=1$; the converse
follows by substitution.  For
$\lceil r/2\rceil\leq s\leq t\leq r$, with
$c\in\mathsf U_{r-s}$ and $d\in\mathsf U_{r-t}$, the product formula gives
\[
 X_{s,1,c}X_{t,1,d}
 =X_{t,1,d}X_{s,1,c}
 =
 \begin{cases}
  X_{s,1,c},&c\equiv d\pmod{p^{r-t}},\\
  0,&\text{otherwise}
 \end{cases}.
\]
Thus all idempotents commute, including zero, and the cited criterion
shows that $S_r$ is a finite inverse semigroup.

The complex algebra of a finite inverse semigroup is semisimple by
\cite[Corollary~4.7]{SteinbergMobiusII}.  The algebra $E_r$ is its
quotient by the one-dimensional ideal spanned by the basis vector
corresponding to the absorbing zero, and is therefore semisimple.
\end{proof}

\begin{theorem}
\label{thm: cyclic factorization decomposition}
For $G_n=C_{p^n}$, there is a noncanonical algebra isomorphism
\begin{equation}
\label{eq: cyclic prime power decomposition}
 \mathbb C\otimes_{\mathbb Z}R_{C_{p^n}}
 \cong
 \operatorname{Mat}_{n+1}(\mathbb C)
 \times
 \prod_{j=1}^n
 \operatorname{Mat}_{n-j+1}(E_j).
\end{equation}
Consequently, $\mathbb C\otimes_{\mathbb Z}R_{C_N}$ is semisimple for
every finite cyclic group $C_N$.
\end{theorem}

\begin{proof}
For $n=0$, the formula is the bottom case
\eqref{eq: cyclic bottom matrix algebra}.  Assume $n\geq1$ and put
$\mathfrak r=\operatorname{Jac}(\mathcal F_n)$.  Its image in the last
quotient in the factorization chain,
$\mathcal F_n/\mathcal F_{n-1}\cong E_n$, is zero.  Hence
$\mathfrak r\subseteq\mathcal F_{n-1}$.  More generally, if
$\mathfrak r\subseteq\mathcal F_r$, its image in
$\mathcal F_r/\mathcal F_{r-1}$ is a nilpotent two-sided ideal in a
semisimple algebra and is zero.  Thus $\mathfrak r$ descends to
$\mathcal F_0$, which is semisimple by
\eqref{eq: cyclic bottom matrix algebra}, and $\mathfrak r=0$.
Every ideal in a
finite-dimensional semisimple algebra has a complementary ideal.
Successively splitting the factorization chain and using
Proposition~\ref{prop: cyclic matrix layers} together with
\eqref{eq: cyclic bottom matrix algebra} gives
\eqref{eq: cyclic prime power decomposition}.

For arbitrary $N$, the tensor-product decomposition at the beginning of
this subsection expresses $R_{C_N}$ as a tensor product of the rings for its
cyclic Sylow subgroups.  The tensor product of finite-dimensional semisimple
complex algebras is semisimple, proving the general cyclic conclusion.
\end{proof}

\begin{corollary}
\label{cor: full semisimplicity cyclicity}
For a finite group $G$, the following conditions are equivalent:
\begin{enumerate}
\item[\textup{(i)}] $G$ is cyclic;
\item[\textup{(ii)}] $\mathbb C\otimes_{\mathbb Z}J_G$ is semisimple;
\item[\textup{(iii)}] $\mathbb C\otimes_{\mathbb Z}R_G$ is semisimple.
\end{enumerate}
\end{corollary}

\begin{proof}
The equivalence of \textup{(i)} and \textup{(ii)} is
Theorem~\ref{thm: bottom cyclicity criterion}, and
\textup{(i)} implies \textup{(iii)} by
Theorem~\ref{thm: cyclic factorization decomposition}.
Finally, \eqref{eq: bottom radical in full radical} shows that
\textup{(iii)} implies \textup{(ii)}.
\end{proof}

\section{Brauer--Picard orbits and the Morita class of $\Vec_G$}
\label{sect: BrPic orbits Morita class}

The distinguished basis $\Omega_G$ parametrizes indecomposable
$\Vec_G$-module categories.  For $a=[K,\nu]\in\Omega_G$, write
\[
 \mathcal M_a=\mathcal M(K,\nu)
 =\operatorname{RMod}_{A(K,\nu^{-1})}(\Vec_G)
\]
and define the dual fusion category
\[
 \mathcal C_a
 =
 \operatorname{Fun}_{\Vec_G}(\mathcal M_a,\mathcal M_a).
\]
This category is Morita equivalent to $\Vec_G$.  By \cite[Proposition~3.6 and Remark~3.7]{MN}, one has
\[
 \mathcal C_a\simeq_\otimes\mathcal C_b
 \quad\Longleftrightarrow\quad
 a\text{ and }b\text{ lie in the same }
 \BrPic(\Vec_G)\text{-orbit}.
\]
Thus
\[
 \BrPic(\Vec_G)\backslash\Omega_G
\]
parametrizes the tensor-equivalence classes of fusion categories in the
Morita class of $\Vec_G$.  Write
$o(G)=\#\bigl(\BrPic(\Vec_G)\backslash\Omega_G\bigr)$.
We first give general orbit results and characterize transitivity.  For
abelian $G$, we obtain lower bounds for the number of orbits, determine all
possible Lagrangian types, solve the orbit problem for homocyclic
$p$-groups, and show that abstract type is not complete for finite abelian groups with mixed exponents.

\subsection{General orbit results}
\label{subsect: orbit underlying groups}

For a finite group $H$, let $\Rad_{\mathrm{sol}}(H)$ denote its solvable
radical.  Put $R=\Rad_{\mathrm{sol}}(G)$ and $\overline G=G/R$, and let
$\pi:G\to\overline G$ be the quotient map.  We write cohomology classes
additively in this subsection, while cocycle formulas remain
multiplicative.

\begin{proposition}
\label{prop: BrPic solvable quotient action}
There is a canonical homomorphism
\[
 \rho_G^{\mathrm{sol}}:\BrPic(\Vec_G)\longrightarrow\Out(\overline G)
\]
for which the assignment taking $[K,\nu]\in\Omega_G$ to the
$\overline G$-conjugacy class of $\pi(K)$ is equivariant.  Thus, if
$u\triangleright[K,\nu]=[K',\nu']$ and
$\sigma\in\Aut(\overline G)$ represents $\rho_G^{\mathrm{sol}}(u)$, then $\pi(K')$ is
conjugate in $\overline G$ to $\sigma(\pi(K))$.
\end{proposition}

\begin{proof}
Write $u=[L,\mu]$.  By Theorem~\ref{thm: invertibility criterion},
$L$ is subdirect and its coordinate kernels $A_i=k_i(L)$ are abelian
normal subgroups of $G$.  Goursat's isomorphism
$\phi:G/A_1\xrightarrow{\sim}G/A_2$ is characterized by
$(x,y)\in L$ if and only if $\phi(xA_1)=yA_2$.
Both $A_i$ lie in $R$.  For any solvable normal subgroup
$B\triangleleft H$, one has
$\Rad_{\mathrm{sol}}(H/B)=\Rad_{\mathrm{sol}}(H)/B$: the indicated
image is solvable and normal, and the preimage of every solvable normal
subgroup of $H/B$ is a solvable normal extension of $B$.
Consequently, $\phi$ carries $R/A_1$ onto $R/A_2$ and induces
$\overline\phi\in\Aut(\overline G)$.

Put $\sigma_L=\overline\phi^{-1}$.  Then
\[
 (\pi\times\pi)(L)
 =\{(\sigma_L(t),t):t\in\overline G\}.
\]
Conjugation of $L$ by $(a,b)\in G\times G$ replaces $\sigma_L$ by
$c_{\pi(a)}\circ\sigma_L\circ c_{\pi(b)}^{-1}$, where $c_z$
denotes conjugation by $z$.  Thus $\rho_G^{\mathrm{sol}}([L,\mu])=[\sigma_L]$
is well defined; changing the cocycle does not change this construction.

Let $v=[M,\tau]$.  Since $p_2(L)=p_1(M)=G$, the Clifford formula
gives a representative $uv=[J,\lambda]$ with $J\leq L*M$.
Theorem~\ref{thm: invertibility criterion} implies that $J$ is
subdirect.  Its image in $\overline G\times\overline G$ is therefore
a subdirect subgroup of the graph of $\sigma_L\circ\sigma_M$,
and hence equals that graph.  It follows that
$\rho_G^{\mathrm{sol}}(uv)=\rho_G^{\mathrm{sol}}(u)\rho_G^{\mathrm{sol}}(v)$; the diagonal subgroup gives the
identity.

For equivariance, Corollary~\ref{thm: one-sided Clifford action}
allows a representative of $[K',\nu']$ with
\[
 K'\leq p_1\bigl(L\cap(G\times K)\bigr).
\]
Hence $\pi(K')\leq\sigma_L(\pi(K))$, so
$|\pi(K')|\leq|\pi(K)|$.  Apply the same containment to
$u^{-1}\triangleright[K',\nu']=[K,\nu]$.
The representative obtained may be conjugate to $K$, but its image
has the same order as $\pi(K)$.  Thus
$|\pi(K)|\leq|\pi(K')|$, and the first containment is an equality
for the chosen representatives.  This proves the assertion on
conjugacy classes.
\end{proof}

\begin{corollary}
\label{cor: semisimple core orbit invariant}
If $[K,\nu]$ and $[K',\nu']$ belong to the same
$\BrPic(\Vec_G)$-orbit, then
\[
 K/\Rad_{\mathrm{sol}}(K)
 \cong K'/\Rad_{\mathrm{sol}}(K').
\]
In particular, solvability of the supporting subgroup and the multiset
of its nonabelian composition factors are constant on each orbit.
\end{corollary}

\begin{proof}
The subgroup $K\cap R$ is solvable and normal in $K$, and
$\pi(K)\cong K/(K\cap R)$.  The radical-quotient identity proved above
therefore gives
\[
 \pi(K)/\Rad_{\mathrm{sol}}(\pi(K))
 \cong K/\Rad_{\mathrm{sol}}(K).
\]
Apply Proposition~\ref{prop: BrPic solvable quotient action} to $K$ and
$K'$.  Automorphisms preserve solvable radicals, so these quotients are
isomorphic.  Solvable normal subgroups contribute only abelian
composition factors, giving the last assertion.
\end{proof}

Thus the isomorphism type of $K/\Rad_{\mathrm{sol}}(K)$ is an invariant
of the tensor-equivalence class of the corresponding dual fusion category.

\begin{corollary}
\label{cor: orbit count solvable quotient}
Let $\operatorname{Sub}(\overline G)$ denote the set of subgroups of
$\overline G$.  Then
\[
 o(G)\geq
 \#\bigl(\Aut(\overline G)\backslash
              \operatorname{Sub}(\overline G)\bigr).
\]
\end{corollary}

\begin{proof}
By Proposition~\ref{prop: BrPic solvable quotient action}, the
assignment taking $[K,\nu]$ to the $\Aut(\overline G)$-orbit of
$\pi(K)$ is constant on Brauer--Picard orbits.  It is surjective,
since $H\leq\overline G$ is the image of $\pi^{-1}(H)\leq G$.
\end{proof}

In particular, distinct subgroup orders occurring in $\overline G$ give
distinct Brauer--Picard orbits.

\begin{corollary}
\label{cor: radical free orbit criterion}
Suppose $\Rad_{\mathrm{sol}}(G)=1$.  Then $[K,\nu]$ and $[K',\nu']$
belong to the same Brauer--Picard orbit if and only if there exists
$\alpha\in\Aut(G)$ with $\alpha(K)=K'$ such that
\[
 [\nu']-\alpha_*[\nu]\in
 \operatorname{im}\!\left(
 H^2(G,k^\times)\xrightarrow{\operatorname{res}}H^2(K',k^\times)
 \right).
\]
Here transport is defined by
$(\alpha_*\nu)(\alpha(k),\alpha(l))=\nu(k,l)$ for $k,l\in K$.
\end{corollary}

\begin{proof}
Every abelian normal subgroup of $G$ is trivial.  Hence, by
Theorem~\ref{thm: invertibility criterion}, if
$[L,\mu]\in\BrPic(\Vec_G)$, then $k_1(L)=k_2(L)=1$, so
\[
 L=L_\alpha=\{(\alpha(g),g):g\in G\}
\]
for some $\alpha\in\Aut(G)$.  Since $p_1:L_\alpha\to G$ is an
isomorphism, we may choose $\mu=p_1^*\xi$ with
$\xi\in Z^2(G,k^\times)$.  The one-sided formula gives
\[
 [L_\alpha,p_1^*\xi]\triangleright[K,\nu]
 =\bigl[\alpha(K),\,
       \alpha_*[\nu]+\operatorname{res}^{G}_{\alpha(K)}[\xi]\bigr].
\]

Suppose that $[K,\nu]$ and $[K',\nu']$ lie in the same orbit.
For some $\alpha_0\in\Aut(G)$ and $[\xi]\in H^2(G,k^\times)$,
the pair in the displayed formula is $G$-conjugate to
$(K',[\nu'])$.  Thus there is $g\in G$ such that
$K'=c_g\alpha_0(K)$.  Put $\alpha=c_g\circ\alpha_0$.
Transport commutes with restriction, and inner automorphisms act
trivially on $H^2(G,k^\times)$, so
\[
 \begin{aligned}
 [\nu']
 &=(c_g)_*\bigl(\alpha_{0*}[\nu]
              +\operatorname{res}^{G}_{\alpha_0(K)}[\xi]\bigr)\\
 &=\alpha_*[\nu]+\operatorname{res}^{G}_{K'}[\xi].
 \end{aligned}
\]
This proves the required condition.

Conversely, suppose that $\alpha(K)=K'$ and
$[\nu']-\alpha_*[\nu]=\operatorname{res}^{G}_{K'}[\xi]$ for
some $[\xi]\in H^2(G,k^\times)$.  The decorated graph
$[L_\alpha,p_1^*\xi]$ belongs to $\BrPic(\Vec_G)$ by
Theorem~\ref{thm: invertibility criterion}, and the first displayed
formula shows that it sends $[K,\nu]$ to $[K',\nu']$.

The graph description also gives
$\BrPic(\Vec_G)\cong H^2(G,k^\times)\rtimes\Out(G)$;
see \cite[Corollary~7.8]{NR}.
\end{proof}

\begin{example}
\label{ex: S5 subgroup embedding orbits}
For $G=S_5$, the solvable radical is trivial and every automorphism is
inner.  The subgroups $K=\langle(12)\rangle$
and $K'=\langle(12)(34)\rangle$ are not conjugate, so $[K,1]$ and
$[K',1]$ lie in distinct Brauer--Picard orbits.  Yet both subgroups are
isomorphic to $C_2$ and have trivial quotients by their solvable
radicals.  Thus the image of the subgroup in the ambient quotient
distinguishes these two embeddings, while the abstract solvable-radical
quotient does not.
\end{example}

\begin{remark}
\label{rem: radical free orbit counts}
If $\Rad_{\mathrm{sol}}(G)=1$, the graph description gives
\[
 o(G)\geq
 \left\lceil
 \frac{|\Omega_G|}
 {|H^2(G,k^\times)|\,|\Out(G)|}
 \right\rceil,
\]
since every orbit has at most $|\BrPic(\Vec_G)|$ elements.
If also $H^2(G,k^\times)=0$ and $\Out(G)=1$, then
$\BrPic(\Vec_G)=1$ and every orbit is a singleton.
This occurs, for example, for the Mathieu group $M_{11}$ and the
Monster; see \cite[Remark~7.10]{NR}.
\end{remark}

For finite abelian $G$, put
\[
 \mathbb H(G)=G\oplus\widehat G,
 \qquad
 q(x,\chi)=\chi(x).
\]
The center $\Z(\Vec_G)$ is the pointed nondegenerate braided category with
simple objects $(x,\chi)\in\mathbb H(G)$.  We use the convention in which
the half-braiding of $(x,\chi)$ with $\delta_y$ is
$\chi(y)\id$.  Together with the convention
\[
 \mathcal M(K,\nu)
 =
 \operatorname{RMod}_{A(K,\nu^{-1})}(\Vec_G),
\]
this makes the inverse alternating bicharacter appear below.  Let
$O(\mathbb H(G),q)$ denote the group of automorphisms of $\mathbb H(G)$
preserving $q$.  The standard identifications give
\[
 \BrPic(\Vec_G)
 \cong
 \Aut^{\mathrm{br}}(\Z(\Vec_G))
 \cong
 O(\mathbb H(G),q);
\]
see \cite[Theorem~1.1 and Corollary~1.2]{ENO2}.

We use the identification sending an invertible $\Vec_G$-bimodule
category $\mathcal U$ to the braided autoequivalence $\Psi_{\mathcal U}$
characterized by
\[
 \Psi_{\mathcal U}(Z)\triangleright-
 \cong-\triangleleft Z
\]
as bimodule endofunctors of $\mathcal U$, naturally in
$Z\in\Z(\Vec_G)$.  Thus
$\Psi_{\mathcal U\boxtimes_{\Vec_G}\mathcal V}
 \cong\Psi_{\mathcal U}\circ\Psi_{\mathcal V}$;
this is the inverse transport to $\Phi_{\mathcal U}$ in
\cite[Equation~(12)]{NR}.

Let
$\operatorname{Lag}(\mathbb H(G),q)$ denote the
subgroups on which $q$ is trivial and whose order is $|G|$.

\begin{proposition}
\label{prop: Omega Lagrangian correspondence}
For finite abelian $G$, the formula
\begin{equation}
\label{eq: abelian Lagrangian support}
 \Lambda(K,\nu)
 =
 \{(x,\chi)\in G\oplus\widehat G:
   x\in K,\ \chi|_K=\Alt_\nu(x,-)^{-1}\}
\end{equation}
defines a $\BrPic(\Vec_G)$-equivariant bijection
\[
 \Omega_G\xrightarrow{\ \sim\ }
 \operatorname{Lag}(\mathbb H(G),q).
\]
\end{proposition}

\begin{proof}
For abelian $G$, every subgroup $K$ is normal and every cohomology class
on $K$ is $G$-invariant.  Thus the classification and module-category
correspondence of
\cite[Equation~(13) and Theorems~3.5--3.6]{NN}
apply to every $[K,\nu]\in\Omega_G$.  In the notation of that source, the
corresponding Lagrangian subcategory has simple objects $(x,\chi)$ with
$x\in K$ and $\chi|_K=B(x,-)$.

Our convention uses the algebra $A(K,\nu^{-1})$.  The resulting
bicharacter is therefore
$B=\Alt_{\nu^{-1}}=\Alt_\nu^{-1}$.  Equivalently, the sign can be checked
directly from the full-center centrality equation: on
$v_x\otimes e_y$ it gives
\[
 \nu(x,y)^{-1}=\chi(y)\nu(y,x)^{-1},
\]
and hence $\chi(y)=\Alt_\nu(x,y)^{-1}$.  This gives
\eqref{eq: abelian Lagrangian support}.  Since $\Z(\Vec_G)$ is pointed,
a Lagrangian subcategory is determined by its subgroup of invertible simple
objects, proving the bijection.  The algebra-object form of the full-center
correspondence is compatible with this support description by
\cite[Proposition~1.3]{DavydovSimmons}.

For equivariance of the module-category/Lagrangian correspondence
with this convention, see \cite[Remark~7.13]{NR} and
\cite[Section~5.1]{ENO2}.
\end{proof}

Proposition~\ref{prop: Omega Lagrangian correspondence}
therefore gives
\[
 \BrPic(\Vec_G)\backslash\Omega_G
 \cong
 O(\mathbb H(G),q)\backslash
 \operatorname{Lag}(\mathbb H(G),q).
\]

For $K\leq G$, let
$K^\perp=\{\chi\in\widehat G:\chi|_K=1\}$.
Then $\Lambda(K,1)=K\oplus K^\perp$.
Characters trivial on $K$ factor through $G/K$, so
$K^\perp\cong\widehat{G/K}$ and
$\Lambda(K,1)\cong K\oplus G/K$ as abstract groups.

\begin{proposition}
\label{prop: BrPic transitivity}
The action of $\BrPic(\Vec_G)$ on $\Omega_G$ is transitive if and
only if $G$ is abelian of square-free exponent.
\end{proposition}

\begin{proof}
For $u=[L,\mu]\in\BrPic(\Vec_G)$, put $A_1=k_1(L)$.
The one-sided formula with input $[1,1]$ has trivial overlap and gives
$u\triangleright[1,1]=[A_1,\mu_1]$, where
$\mu_1(a,a')=\mu((a,1),(a',1))$.
If the action is transitive, some such element takes $[1,1]$ to
$[G,1]$.  The invertibility criterion says that $A_1$ is abelian,
so $G$ is abelian.

If $G$ does not have square-free exponent, write
$G=C_{p^a}\oplus H$ for a prime $p$ and $a\geq2$, and take
$K=pC_{p^a}\leq G$.  Then
\[
 \Lambda(K,1)\cong C_{p^{a-1}}\oplus C_p\oplus H,
 \qquad
 \Lambda(1,1)\cong G=C_{p^a}\oplus H.
\]
These groups are not isomorphic: their $p$-primary components require
different numbers of generators.  Orthogonal automorphisms preserve
abstract group type, so Proposition~\ref{prop: Omega Lagrangian correspondence}
shows that the action is not transitive.

Conversely, suppose that $G$ is abelian of square-free exponent.
Every subgroup has a complement.  Given $[K,\nu]\in\Omega_G$,
choose a normalized representative $\nu$, write $G=K\times H$, and choose
a perfect bicharacter
$b:K\times K\to k^\times$.  On
\[
 L=\{((x,h),(y,h)):x,y\in K,\ h\in H\}\leq G\times G,
\]
write elements as $(x,y,h)$ and set
\[
 \mu\bigl((x,y,h),(x',y',h')\bigr)
 =\nu(x,x')\,b(x,y').
\]
This is a cocycle, being the product of a pulled-back cocycle and a
bicharacter.  The subgroup $L$ is subdirect, its two coordinate kernels
are $K$, and its mixed pairing is the perfect pairing $b$.
Thus $[L,\mu]\in\BrPic(\Vec_G)$ by
Theorem~\ref{thm: invertibility criterion}.
Its restriction to the first kernel is $\nu$, so the one-sided
formula gives $[L,\mu]\triangleright[1,1]=[K,\nu]$.
\end{proof}

\begin{remark}
\label{rem: S3 two orbits}
If $G$ is nonabelian, Proposition~\ref{prop: BrPic transitivity} gives
$o(G)\geq2$.  The smallest example is $G=S_3$, for which there are
exactly two orbits,
\[
 \{[1,1],[C_3,1]\},\qquad
 \{[C_2,1],[S_3,1]\}.
\]
Here $\BrPic(\Vec_{S_3})\cong C_2$ by \cite[Section~8.1]{NR},
and the two transpositions follow from
Corollary~\ref{thm: one-sided Clifford action}.
\end{remark}

\subsection{The abelian case: Lagrangian types}
\label{subsect: orbit abelian orthogonal model}

Assume throughout this subsection that $G$ is finite abelian.
If $G=\bigoplus_pG_p$ is its primary decomposition, then
$\mathbb H(G)=\bigoplus_p\mathbb H(G_p)$ orthogonally.  Every subgroup
and every automorphism preserving $q$ decomposes over the $p$-primary
components, and a subgroup of $\mathbb H(G)$ is Lagrangian if and only if
each of its $p$-primary components is Lagrangian.  Hence
$o(G)=\prod_p o(G_p)$.

\begin{proposition}
\label{prop: abelian orbit count lower bound}
Let $p$ be a prime and let
$G=\bigoplus_{i=1}^rC_{p^{a_i}}$, where $a_i\geq1$.  Then
\[
 o(G)\geq1+\sum_{i=1}^r\left\lfloor\frac{a_i}{2}\right\rfloor.
\]
\end{proposition}

\begin{proof}
For $\mathbf b=(b_1,\ldots,b_r)$ with
$0\leq b_i\leq\lfloor a_i/2\rfloor$, put
$K_{\mathbf b}=\bigoplus_i p^{b_i}C_{p^{a_i}}\leq G$.
By \eqref{eq: abelian Lagrangian support}, with $\nu=1$,
$\Lambda(K_{\mathbf b},1)=K_{\mathbf b}\oplus K_{\mathbf b}^{\perp}$.
Since $K_{\mathbf b}^{\perp}\cong G/K_{\mathbf b}$,
\[
 \Lambda(K_{\mathbf b},1)
 \cong\bigoplus_i
 \left(C_{p^{a_i-b_i}}\oplus C_{p^{b_i}}\right),
\]
with trivial factors omitted.  Starting with all $b_i=0$, increase
one coordinate at a time until each reaches $\lfloor a_i/2\rfloor$.
At a step replacing $b_i$ by $b_i+1$, the sum of the squares of the
exponents in the cyclic decomposition decreases by
\[
 (a_i-b_i)^2+b_i^2
 -(a_i-b_i-1)^2-(b_i+1)^2
 =2(a_i-2b_i-1)>0.
\]
The resulting $1+\sum_i\lfloor a_i/2\rfloor$ groups therefore
have distinct abstract isomorphism types.  Since orthogonal
automorphisms preserve these types,
Proposition~\ref{prop: Omega Lagrangian correspondence} gives the bound.
\end{proof}

We now determine the possible abstract group types of Lagrangian
subgroups of $\mathbb H(G)$.
For a prime $p$ and a partition
$\lambda=(\lambda_1\geq\lambda_2\geq\cdots)$, put
$G_\lambda=\bigoplus_iC_{p^{\lambda_i}}$.  If $\rho$ and $\sigma$ are
partitions, let $\rho\sqcup\sigma$ be the decreasing rearrangement of the
combined multisets of parts.  Thus
$G_{\rho\sqcup\sigma}\cong G_\rho\oplus G_\sigma$.

For a finite abelian group $A$, write $\operatorname{type}(A)$ for its
isomorphism type.  If $A$ is a $p$-group, we identify
$\operatorname{type}(A)$ with the unique partition $\rho$ such that
$A\cong G_\rho$.

For partitions $\rho,\sigma,\gamma$, let
$c_{\rho,\sigma}^{\gamma}$ denote the Littlewood--Richardson coefficient,
that is, the coefficient of $s_\gamma$ in the Schur-function product
$s_\rho s_\sigma$; see
\cite[Section~3]{FultonEigenvalues}.  The Green--Klein theorem
\cite{Klein} says that
$c_{\rho,\sigma}^{\gamma}>0$ if and only if there is an exact sequence
\[
 0\longrightarrow G_\rho
 \longrightarrow G_\gamma
 \longrightarrow G_\sigma
 \longrightarrow0.
\]

We recall the precise form of the Horn criterion used below.  For an
$r$-element subset
$I=\{i_1<\cdots<i_r\}\subseteq\{1,\ldots,n\}$, put
\[
 I^\sharp=(i_r-r,\ldots,i_2-2,i_1-1),
\]
with zero parts omitted.  If $\rho,\sigma,\gamma$ are partitions with at
most $n$ parts, then $c_{\rho,\sigma}^{\gamma}>0$ if and only if
\[
 |\rho|+|\sigma|=|\gamma|
\]
and
\[
 \sum_{t\in T}\gamma_t
 \leq
 \sum_{i\in I}\rho_i+\sum_{j\in J}\sigma_j
\]
for every $1\leq r<n$ and all $r$-element subsets
$I,J,T\subseteq\{1,\ldots,n\}$ satisfying
\[
 c_{I^\sharp,J^\sharp}^{T^\sharp}>0.
\]
This is the Horn criterion for Littlewood--Richardson positivity
\cite[Theorems~11--12]{FultonEigenvalues}.

\begin{lemma}
\label{lem: odd even LR reduction}
Let $\lambda$ and $\mu$ be partitions such that
\[
 c_{\mu,\mu}^{\lambda\sqcup\lambda}>0.
\]
Choose $n$ so that $\lambda$ has at most $n$ parts and $\mu$ has at most
$2n$ parts, and append zeros so that
\[
 \lambda=(\lambda_1,\ldots,\lambda_n),
 \qquad
 \mu=(\mu_1,\ldots,\mu_{2n}).
\]
Put
\[
 \alpha=(\mu_1,\mu_3,\ldots,\mu_{2n-1}),
 \qquad
 \beta=(\mu_2,\mu_4,\ldots,\mu_{2n}).
\]
Then $\alpha$ and $\beta$ are partitions,
$\alpha\sqcup\beta=\mu$, and
\[
 c_{\alpha,\beta}^{\lambda}>0.
\]
\end{lemma}

\begin{proof}
Since $\mu$ is weakly decreasing, so are $\alpha$ and $\beta$, and their
combined multiset of parts is the multiset of parts of $\mu$.  Hence
$\alpha\sqcup\beta=\mu$.  Since Schur functions are homogeneous, the
hypothesis gives
\[
 2|\mu|
 =
 |\lambda\sqcup\lambda|
 =
 2|\lambda|.
\]
Consequently,
\[
 |\alpha|+|\beta|=|\mu|=|\lambda|.
\]

Fix $1\leq r<n$ and $r$-element subsets
\[
 I=\{i_1<\cdots<i_r\},\qquad
 J=\{j_1<\cdots<j_r\},\qquad
 T=\{t_1<\cdots<t_r\}
\]
of $\{1,\ldots,n\}$ such that
$c_{I^\sharp,J^\sharp}^{T^\sharp}>0$.  Define
\[
 F=
 \{2i-1:i\in I\}\cup\{2j:j\in J\},
 \qquad
 D=
 \{2t-1:t\in T\}\cup\{2t:t\in T\}.
\]
By \cite[Lemma~1.18]{FFLP},
\[
 c_{F^\sharp,F^\sharp}^{D^\sharp}>0.
\]
Since $2r<2n$, the corresponding Horn inequality for
$c_{\mu,\mu}^{\lambda\sqcup\lambda}>0$ is
\[
 \sum_{d\in D}(\lambda\sqcup\lambda)_d
 \leq
 \sum_{f\in F}\mu_f+\sum_{f\in F}\mu_f.
\]
Using
\[
 (\lambda\sqcup\lambda)_{2t-1}
 =
 (\lambda\sqcup\lambda)_{2t}
 =
 \lambda_t,
 \qquad
 \mu_{2i-1}=\alpha_i,
 \qquad
 \mu_{2j}=\beta_j,
\]
the preceding Horn inequality is equivalent to
\[
 \sum_{t\in T}\lambda_t
 \leq
 \sum_{i\in I}\alpha_i+\sum_{j\in J}\beta_j.
\]
Since these subsets were arbitrary and the required size equality has
already been verified, the Horn criterion gives
$c_{\alpha,\beta}^{\lambda}>0$.
\end{proof}

\begin{theorem}
\label{thm: possible Lagrangian types}
Let $G$ be a finite abelian group.  The abstract group types of Lagrangian
subgroups of $\mathbb H(G)$ are precisely
\[
 \{\operatorname{type}(K\oplus G/K):K\leq G\}.
\]
If $G=G_\lambda$ is a $p$-group, this set is identified with
\[
 \{\alpha\sqcup\beta:c_{\alpha,\beta}^{\lambda}>0\}.
\]
\end{theorem}

\begin{proof}
By primary decomposition, it suffices to prove the theorem for
$G=G_\lambda$, a finite abelian $p$-group.

For $K\leq G$, the observation above gives, noncanonically as abstract
finite abelian groups, $\Lambda(K,1)\cong K\oplus G/K$.  Thus every type
in the stated set occurs.

Conversely, let $\Lambda\leq\mathbb H(G)$ be Lagrangian and let
$\mu=\operatorname{type}(\Lambda)$.  Since $\Lambda$ is Lagrangian,
the bicharacter associated to the hyperbolic form induces an isomorphism
\[
 \mathbb H(G)/\Lambda\xrightarrow{\ \sim\ }\widehat\Lambda.
\]
A finite abelian group and its character group have the same partition
type, so there is an exact sequence
\[
 0\longrightarrow\Lambda
 \longrightarrow\mathbb H(G_\lambda)
 \longrightarrow\widehat\Lambda
 \longrightarrow0.
\]
The three terms have types $\mu$, $\lambda\sqcup\lambda$, and $\mu$,
respectively.  By the Green--Klein theorem
\cite{Klein},
\begin{equation}
\label{eq: doubled LR positivity}
 c_{\mu,\mu}^{\lambda\sqcup\lambda}>0.
\end{equation}

Lemma~\ref{lem: odd even LR reduction} gives partitions
$\alpha$ and $\beta$ such that
\[
 \alpha\sqcup\beta=\mu,
 \qquad
 c_{\alpha,\beta}^{\lambda}>0.
\]
By Green--Klein, there is a subgroup $K\leq G_\lambda$ such that
\[
 \operatorname{type}(K)=\alpha,
 \qquad
 \operatorname{type}(G_\lambda/K)=\beta.
\]
Therefore
\[
 \operatorname{type}(K\oplus G/K)
 =
 \alpha\sqcup\beta
 =
 \mu
 =
 \operatorname{type}(\Lambda),
\]
which proves the reverse inclusion.
\end{proof}

\begin{corollary}
\label{cor: Lagrangian subcategory types}
The symmetric equivalence types of Lagrangian subcategories of
$\Z(\Vec_G)$ are precisely
\[
 \Rep(K\oplus G/K),
 \qquad
 K\leq G.
\]
\end{corollary}

\begin{proof}
A Lagrangian subgroup $\Lambda$ determines a symmetric pointed subcategory
equivalent to $\Rep(\widehat\Lambda)$.  Since
$\widehat\Lambda\cong\Lambda$ abstractly, the assertion follows from
Theorem~\ref{thm: possible Lagrangian types}.
\end{proof}

In the classification of \cite{NN}, the Tannakian group $G'$ may be a
nonsplit extension
\[
 0\longrightarrow\widehat H
 \longrightarrow G'
 \longrightarrow G/H
 \longrightarrow0
\]
for a subgroup $H\leq G$.  Theorem~\ref{thm: possible Lagrangian types}
shows nevertheless that, for some $K\leq G$,
$G'\cong K\oplus G/K$ as an abstract group.

\subsection{Homocyclic groups and mixed exponents}
\label{subsect: orbit homocyclic mixed}

For odd $p$, Mason and Ng proved an important special case of the orbit
statement below.  Their Theorem~13.7(ii) implies that, in a fixed metabolic
form, metabolizers isomorphic to $C_{p^n}^k$ form a single orbit under the
isometry group \cite[Theorem~13.7(ii)]{MasonNg}.  In the terminology of
this paper, these metabolizers are Lagrangian subgroups.
Theorem~\ref{thm: homocyclic Lagrangian orbits} classifies all Lagrangian
subgroups of the hyperbolic form associated to $C_{p^n}^k$: their abstract
group type is a complete orbit invariant.  It also applies for $p=2$.
The finite-local-ring orthogonal classification used in the proof is due
to Zimmermann
\cite[Theorem~2.4.4 and Corollaries~2.4.7, 2.4.9]{Zimmermann}.

Fix a prime $p$, put $R_n=\mathbb Z/p^n\mathbb Z$, and let
$G=R_n^k\cong C_{p^n}^k$.  After choosing a compatible primitive
$p^n$-th root of unity, identify
$\mathbb H(G)\cong R_n^k\oplus R_n^k$ with the standard hyperbolic
quadratic module.  Choose a hyperbolic basis
$e_1,\ldots,e_k,f_1,\ldots,f_k$ and write its $R_n$-valued quadratic form
as
\[
 Q_n\left(\sum_i x_ie_i+y_if_i\right)=\sum_i x_iy_i.
\]
For a nondecreasing tuple
$0\leq a_1\leq\cdots\leq a_k\leq\lfloor n/2\rfloor$, define
\begin{equation}
\label{eq: homocyclic normal Lagrangian}
 \Lambda_{\mathbf a}
 =
 \bigoplus_{i=1}^k
 \left(
 p^{a_i}R_ne_i
 \oplus
 p^{n-a_i}R_nf_i
 \right),
\end{equation}
where $p^nR_n=0$.

\begin{theorem}
\label{thm: homocyclic Lagrangian orbits}
Every Lagrangian subgroup of $\mathbb H(C_{p^n}^k)$ is orthogonally
conjugate to a unique $\Lambda_{\mathbf a}$.  Consequently,
\begin{equation}
\label{eq: homocyclic orbit parametrization}
 \BrPic(\Vec_{C_{p^n}^k})\backslash\Omega_{C_{p^n}^k}
 \cong
 \left\{
 (a_1,\ldots,a_k):
 0\leq a_1\leq\cdots\leq a_k
 \leq\left\lfloor\frac n2\right\rfloor
 \right\}.
\end{equation}
Equivalently, the abstract group type of the associated Lagrangian subgroup
is a complete orbit invariant.
\end{theorem}

\begin{proof}
The ring $R_n$ is a finite local principal ideal ring with maximal ideal
$(p)$ and nilpotency index $n$, and the ambient module is a free regular
hyperbolic quadratic $R_n$-module of rank $2k$.  Under the chosen
identification, a Lagrangian subgroup is a maximal totally isotropic
$R_n$-submodule, and Zimmermann's block $H_{e_i,f_i}(a_i)$ is precisely
$p^{a_i}R_ne_i\oplus p^{n-a_i}R_nf_i$.  The decomposition, the
nondecreasing parameter range, and transitivity of
the full orthogonal group on submodules of fixed type now follow from
\cite[Theorem~2.4.4 and Corollaries~2.4.7, 2.4.9]{Zimmermann}.

The abstract group type is
\[
 \Lambda_{\mathbf a}
 \cong
 \bigoplus_{i=1}^k
 \left(C_{p^{n-a_i}}\oplus C_{p^{a_i}}\right),
\]
with $C_{p^0}$ omitted.  Since $0\leq a_i\leq n/2$, this type recovers the
multiset $\{a_i\}$ and hence the nondecreasing tuple.  The equivariant
bijection of Proposition~\ref{prop: Omega Lagrangian correspondence}
translates the orthogonal classification into
\eqref{eq: homocyclic orbit parametrization}.
\end{proof}

\begin{corollary}
\label{cor: homocyclic orbit count}
One has
\[
 \#\bigl(
 \BrPic(\Vec_{C_{p^n}^k})\backslash\Omega_{C_{p^n}^k}
 \bigr)
 =
 \binom{k+\lfloor n/2\rfloor}{k}.
\]
\end{corollary}

\begin{proof}
This is the number of multisets of size $k$ chosen from
$\lfloor n/2\rfloor+1$ possible values.
\end{proof}

\begin{corollary}
\label{cor: homocyclic undecorated representatives}
Every $\BrPic(\Vec_{C_{p^n}^k})$-orbit in $\Omega_{C_{p^n}^k}$ contains an
undecorated representative.
\end{corollary}

\begin{proof}
Put
\[
 K_{\mathbf a}
 =
 \bigoplus_{i=1}^k p^{a_i}R_ne_i
 \leq G.
\]
Then
$K_{\mathbf a}^\perp
=\bigoplus_{i=1}^k p^{n-a_i}R_nf_i$, and therefore
$\Lambda(K_{\mathbf a},1)=\Lambda_{\mathbf a}$.
\end{proof}

An orbit may contain several undecorated representatives.  Passing
from a given pair $[K,\nu]$ to $[K_{\mathbf a},1]$ may change the
subgroup $K$.

By \cite[Proposition~3.6 and Remark~3.7]{MN}, the binomial coefficient in
Corollary~\ref{cor: homocyclic orbit count} also
counts tensor-equivalence classes of fusion categories in the Morita class
of $\Vec_{C_{p^n}^k}$.  An undecorated module representative need not have
an untwisted pointed dual category.

\begin{remark}
Barthel's work on symmetric bilinear forms on finite abelian groups gives
useful context for the additional phenomena that occur for mixed exponents
\cite{Barthel}.  Theorem~3.2.1 and
Corollary~3.2.2 give a complete classification in the homogeneous case,
Theorem~3.3.2 gives a substantial structural treatment for two adjacent
exponent levels, and Section~3.4 exhibits new indecomposable singular
phenomena once exponent levels are separated, already at exponent gap two.
\end{remark}

For a general finite abelian $p$-group, every subgroup $p^r\mathbb H(G)$
is characteristic in the underlying abelian group.  Consequently
\[
 \operatorname{type}
 \bigl(\Lambda\cap p^r\mathbb H(G)\bigr)
\]
is an orthogonal-orbit invariant for every $r\geq0$.  The case $r=1$
already shows that abstract type is not complete for mixed exponents.

Mason and Ng give concrete examples of nonconjugate metabolizers in
\cite[Section~14, especially Example~14.1]{MasonNg}; in their
Example~14.1 the metabolizers have different abstract types.
Example~\ref{ex: mixed same type nonconjugate Lagrangians} shows that for
mixed exponents even the abstract type of the Lagrangian need not determine
its orthogonal orbit.

\begin{example}
\label{ex: mixed same type nonconjugate Lagrangians}
Fix an odd prime $p\equiv1\pmod4$, for example $p=5$, and put
\[
 G=C_{p^4}\oplus C_{p^3}\oplus C_{p^2}\oplus C_p,
 \qquad
 E=\mathbb H(G).
\]
Temporarily write the quadratic form additively as
$Q:E\to\mathbb Q/\mathbb Z$, so that the multiplicative form is
$\exp(2\pi iQ)$.  Let $B(u,v)=Q(u+v)-Q(u)-Q(v)$ be the associated
$\mathbb Q/\mathbb Z$-valued bilinear form.  For $n=1,2,3,4$, let
$e_n,f_n$ be a standard hyperbolic pair of order $p^n$, with
\[
 Q(ae_n+bf_n)=\frac{ab}{p^n}\pmod{\mathbb Z}.
\]
Choose $\epsilon\in(\mathbb Z/p^4\mathbb Z)^\times$ satisfying
$\epsilon^2=-1$ and use its reductions modulo $p^n$.  Put
\[
 x_n=e_n+f_n,
 \qquad
 y_n=\epsilon(e_n-f_n).
\]
Then $Q(x_n)=Q(y_n)=1/p^n$, $B(x_n,y_n)=0$, and the summands
$\langle e_n,f_n\rangle$, $n=1,2,3,4$, are mutually orthogonal.

Define
\[
 \Lambda_1
 =
 \left\langle
 p^2x_4,\ p^2y_4,\
 x_3+\epsilon y_3,\
 x_2+\epsilon y_2,\
 x_1+\epsilon y_1
 \right\rangle
\]
and
\[
 \Lambda_2
 =
 \left\langle
 px_4+\epsilon x_2,\
 p^2y_4,\
 px_3+\epsilon x_1,\
 py_3+\epsilon y_1,\
 py_2
 \right\rangle.
\]

Using the mutual orthogonality of the summands
$\langle e_n,f_n\rangle$ and $B(x_n,y_n)=0$, a direct calculation shows
that the displayed generators of both $\Lambda_1$ and $\Lambda_2$ are
isotropic and pairwise orthogonal.  In each case their orders form the
multiset $\{p^3,p^2,p^2,p^2,p\}$, and they are independent, as may be
checked componentwise using that $(x_n,y_n)$ is a basis of
$\langle e_n,f_n\rangle$ (since $-2\epsilon$ is a unit).  Hence
\[
 |\Lambda_1|=|\Lambda_2|=p^{10}=|G|,
\]
so both are Lagrangian and
\[
 \Lambda_1\cong\Lambda_2
 \cong C_{p^3}\oplus C_{p^2}^{\,3}\oplus C_p.
\]

For the first subgroup,
\[
 \Lambda_1\cap pE
 =
 \left\langle
 p^2x_4,\ p^2y_4,\
 p(x_3+\epsilon y_3),\
 p(x_2+\epsilon y_2)
 \right\rangle
 \cong
 C_{p^2}^{\,3}\oplus C_p.
\]
For the second, write
\[
 u=px_4+\epsilon x_2,\quad
 v=p^2y_4,\quad
 w=px_3+\epsilon x_1,\quad
 z=py_3+\epsilon y_1,\quad
 t=py_2.
\]
An element $au+bv+cw+dz+et$ lies in $pE$ exactly when
$a\equiv c\equiv d\equiv0\pmod p$; the obstructions are respectively the
$x_2,x_1,y_1$ components.  Therefore
\[
 \Lambda_2\cap pE
 =
 \langle pu,v,pw,pz,t\rangle
 \cong
 C_{p^2}^{\,2}\oplus C_p^{\,3}.
\]
Since $pE$ is characteristic, an orthogonal conjugacy would preserve the
abstract type of the intersection with $pE$.  The two types differ, so
$\Lambda_1$ and $\Lambda_2$ are not orthogonally conjugate.
\end{example}

\begin{remark}
For finite abelian groups $G$ with mixed exponents, the symmetric equivalence type of a Lagrangian subcategory
of $\Z(\Vec_G)$ need not determine its $\BrPic(\Vec_G)$-orbit.  Indeed,
the two Lagrangian subgroups in
Example~\ref{ex: mixed same type nonconjugate Lagrangians} have the same
abstract group type and hence give symmetrically equivalent Tannakian
categories, but they lie in distinct orthogonal, and therefore
Brauer--Picard, orbits.
\end{remark}

\bibliographystyle{amsalpha}
\bibliography{dnrefs,burnsiderefs}

\end{document}